\documentclass[12pt]{amsart}
\usepackage{amsmath,eucal,verbatim,amsopn,amsthm,amsfonts,amssymb,mathdots,mathrsfs,esint,mathtools,enumitem,mnsymbol,bbm}
\usepackage[margin=1in]{geometry}
\let\mathcal=\CMcal
\allowdisplaybreaks[4]

\def\Alg{\operatorname{Alg}}

\newcommand{\setdiff}{\ensuremath{-}}
\newcommand{\setof}[2]{\ensuremath{\{ #1 : #2 \}}}
\newcommand{\isomorphic}{\ensuremath{\cong}}

\newcommand{\mathe}{\ensuremath{\mathrm{e}}}

\def\Hom{\mathrm{Hom}}

\newcommand{\nonarchimedean}{non-Archimedean} 
\newcommand{\archimedean}{Archimedean} 
\newcommand{\GL}{\ensuremath{\mathrm{GL}}}
\newcommand{\Sp}{\ensuremath{\mathrm{Sp}}}

\newcommand{\GSpin}{\ensuremath{\mathrm{GSpin}}}
\newcommand{\SO}{\ensuremath{\mathrm{SO}}}
\newcommand{\Spin}{\ensuremath{\mathrm{Spin}}}
\newcommand{\rmodulo}[2]{#1 / #2} 
\newcommand{\glnidx}{\ensuremath{n}}

\newcommand{\lmodulo}[2]{\ensuremath{#1 \backslash #2}}
\newcommand{\rconj}[1]{\ensuremath{{}^{#1}}}

\def\Alg{\operatorname{Alg}}
\def\Ind{\operatorname{Ind}}
\def\ind{\operatorname{ind}}
\newcommand{\C}{\ensuremath{\mathbb{C}}}
\newcommand{\R}{\ensuremath{\mathbb{R}}}

\newcommand{\Adele}{\ensuremath{\mathbb{A}}}
\newcommand{\rhs}{right-hand side} 
\newcommand{\lhs}{left-hand side} 

\newtheorem{theo}{Theorem}[section]

\newtheorem{theorem}{Theorem}[section]
\newtheorem{lemma}[theorem]{Lemma}

\newtheorem{corollary}[theorem]{Corollary}
\newtheorem{claim}[theorem]{Claim}
\newtheorem{remark}[theorem]{Remark}
\numberwithin{equation}{section}

\begin{document}
\title[]{The characterization of theta-distinguished representations of $\GL_n$}
\author{Eyal Kaplan}
\address{Department of Mathematics, The Ohio State University, Columbus, OH 43210, USA}
\email{kaplaney@gmail.com}

\maketitle
\begin{abstract}
Let $\theta$ and $\theta'$ be a pair of exceptional representations in the sense of Kazhdan and Patterson \cite{KP},
of a metaplectic double cover of $\GL_n$. The tensor $\theta\otimes\theta'$ is a (very large) representation of $\GL_n$. We characterize
its irreducible generic quotients. In the square-integrable case, these are precisely the representations
whose symmetric square $L$-function has a pole at $s=0$. Our proof of this case involves a new globalization result. In the general case these are the representations induced
from distinguished data or pairs of representations and their contragradients. The combinatorial analysis is based on a complete determination
of the twisted Jacquet modules of $\theta$. As a corollary, $\theta$ is shown to admit a new ``metaplectic Shalika model".
\end{abstract}

\section{Introduction}\label{section:introduction}
Let $F$ be a local \nonarchimedean\ field of characteristic zero. Let
$\theta$ and $\theta'$ be a pair of exceptional representations in the sense of Kazhdan and Patterson \cite{KP}, these are representations
of a metaplectic double cover $\widetilde{\GL}_n=\widetilde{\GL}_n(F)$ of $\GL_n=\GL_n(F)$. The $\GL_n$-module $\theta\otimes\theta'$ has appeared, in both local and global incarnations, in several studies, most notably in the work of
Bump and Ginzburg \cite{BG} on the global symmetric square $L$-function. Locally it has been studied by Savin \cite{Savin3} and Kable \cite{Kable,Kable2},
who considered multiplicity-one properties and spherical quotients. In a more general context, exceptional or minimal representations
have played an important role in the theta correspondence, the descent method and Rankin-Selberg integrals \cite{GRS6,G2,GJS2}.

Our main goal in this work is to characterize the generic quotients of $\theta\otimes\theta'$. An admissible representation $\tau$ of $\GL_n$ is
called distinguished if
\begin{align*}
\Hom_{\GL_n}(\theta\otimes\theta',\tau^{\vee})\ne0,
\end{align*}
where $\tau^{\vee}$ is the representation contragradient to $\tau$.
Here is our main result.
\begin{theo}\label{theorem:intro characterization}
Let $\tau$ be an irreducible generic representation of $\GL_n$.
\begin{enumerate}[leftmargin=*]
\item If $\tau$ is essentially square-integrable, then it is distinguished if and only if it is unitary and the symmetric square $L$-function $L(s,\tau,\mathrm{Sym}^2)$ has a pole at $s=0$.
\item In general write $\tau$ as a parabolically induced representation $\tau=\tau_1\times\ldots\times\tau_m$, where each $\tau_i$ is essentially
square-integrable. Then $\tau$ is distinguished if and only if there is $0\leq m_0\leq\lfloor m/2\rfloor$ such that, perhaps after permuting
the indices of the inducing data, $\tau_{2i}=\tau_{2i-1}^{\vee}$ for $1\leq i\leq m_0$ and $\tau_i$ is distinguished for $2m_0+1\leq i\leq m$.
\end{enumerate}
\end{theo}
The supercuspidal case has already been proved in \cite{me10}. 
The square-integrable case is handled in Theorem~\ref{theorem:supercuspidal dist GLn and pole}. It relies on a new globalization result (see below). In this
case there is one assumption, caused by the globalization method: when $\tau$ is square-integrable with a nontrivial central character and $L(s,\tau,\mathrm{Sym}^2)$ has a pole at $s=0$, the proof that $\tau$ is distinguished relies on the existence of the local functorial lift of Cogdell et al. \cite{CKPS} for quasi-split special even orthogonal groups (the global lift was described in \cite{CPSS}). See Remark~\ref{remark:exact assumptions for globalization lemma} for more details.

The fact that parabolic induction (irreducible or not) of distinguished representations
is also distinguished, i.e., an upper heredity result, was established in \cite{me9}. 
In Theorem~\ref{theorem:tau otimes tau dual is dist} we show that for an irreducible $\tau$ (generic or not), $\tau\times\tau^{\vee}$ is always distinguished.
The structure of generic distinguished representations in general is given in Theorem~\ref{theorem:distinguished irred generic}.

Now consider, for example, the representation ${\nu}^{1/2}\rho\times{\nu}^{-1/2}\rho$, where $\nu=|\det|$ and $\rho$ is an irreducible unitary self-dual supercuspidal representation
of $\GL_n$. It is of length $2$, has a unique irreducible quotient
$\mathrm{Lang}({\nu}^{1/2}\rho\times{\nu}^{-1/2}\rho)$ - the Langlands quotient, and a unique irreducible subrepresentation $\Delta$, which is square-integrable (\cite{Z3} 9.1). Theorem~\ref{theorem:tau otimes tau dual is dist} implies that ${\nu}^{1/2}\rho\times{\nu}^{-1/2}\rho$ is distinguished.
Then Theorem~\ref{theorem:intro characterization} shows that, depending on the pole of $L(s,\Delta,\mathrm{Sym}^2)$ at $s=0$, $\Delta$ is distinguished
or not. But in the absence of a pole, we deduce that $\mathrm{Lang}({\nu}^{1/2}\rho\times{\nu}^{-1/2}\rho)$ is distinguished. This is an example of a
non-generic distinguished representation.

Another immediate corollary of Theorem~\ref{theorem:intro characterization} is that an irreducible generic distinguished representation must be self-dual.

Savin \cite{Savin3} considered the spherical quotients of $\theta\otimes\theta'$.
He conjectured that a spherical representation with a trivial central character is distinguished, if and only if it is the lift of a representation
of a prescribed classical group. The case of $n=3$ was established in \cite{Savin3}. Kable \cite{Kable2} proved that such lifts are distinguished, for a general $n$, using analytic methods. 
The other direction was settled in \cite{me9}, by extending the ideas of \cite{Savin3}.

In more detail, let $B_n=T_n\ltimes N_n$ be the Borel subgroup of $\GL_n$, where $N_n$ is the subgroup of upper triangular unipotent matrices, and $Q_{n-k,k}=M_{n-k,k}\ltimes U_{n-k,k}$ be the standard maximal parabolic subgroup with $M_{n-k,k}\isomorphic\GL_{n-k}\times\GL_k$. In \cite{me9} we described a filtration of
$(\theta\otimes\theta')_{N_n}$ as a $T_n$-module, using the theory of derivatives of Bernstein and
Zelevinsky \cite{BZ1,BZ2}.
The techniques of \cite{me9} break down
in the general setting, where we have a Jacquet module of $\theta\otimes\theta'$ with respect to a unipotent radical of an arbitrary standard parabolic subgroup. The theory of derivatives is no longer applicable, the main barrier being that the action of $M_{n-k,k}$ on the set of
nontrivial characters of $U_{n-k,k}$ is not transitive.

In fact, the arguments here can also be used to deduce the structure of distinguished spherical representations, but the claims in \cite{me9}
are actually stronger: the appearance of the inducing character as a subquotient of $(\theta\otimes\theta')_{N_n}$ determines its combinatorial structure completely.

In this work we take a different approach, and rely on the computation of all the twisted Jacquet modules of $\theta$, corresponding to maximal parabolic
subgroups. Let $0<k<n$ and fix an additive character $\psi$ of $F$. The action of $M_{n-k,k}$ on the characters of $U_{n-k,k}$ has $\min(n-k,k)+1$ orbits, we fix representatives $\psi_j$,
where $0\leq j\leq \min(n-k,k)$. The character $\psi_0$ is the trivial character, and when $n=2k$, $\psi_k$ is the generic character in the sense
that its stabilizer $\mathbb{GL}_k$ in $M_{k,k}$ is reductive. The group $\mathbb{GL}_k$ is simply the diagonal embedding of $\GL_k$ in $M_{k,k}$. Restriction of the cover of $\GL_{2k}$ to $\mathbb{GL}_k$ is a trivial double cover ($\mathbb{GL}_k$ splits under this cover).
\begin{theo}\label{theorem:intro Jacquet modules}
For any $0\leq j\leq \min(n-k,k)$, the Jacquet module of $\theta$ with respect to $U_{n-k,k}$ and $\psi_j$ is isomorphic to
$(\theta_{n-k-j}\widetilde{\otimes}\theta_{k-j})\otimes\gamma_{\psi,(-1)^{j-1}}\psi$.
Here $\theta_{m}$ is an exceptional representation of $\widetilde{\GL}_m$, $\widetilde{\otimes}$ is the metaplectic tensor of Kable \cite{Kable}, $\gamma_{\psi,(-1)^{j-1}}\psi$ is a one-dimensional representation of $\mathbb{GL}_j\ltimes U_{j,j}$, and $\gamma_{\psi,(-1)^{j-1}}$ is a certain Weil factor regarded as a character of $\widetilde{\mathbb{GL}}_j$.
\end{theo}
The main ingredient in the proof is the case of $n=2k=2j$. When $n=2$, $\psi_1$ is the usual Whittaker character of $N_2$ and this result
was proved by Gelbart and Piatetski-Shapiro \cite{GP} (Theorem~2.2). The general case is proved in Theorem~\ref{theorem:twisted Jacquet modules of theta along generic character in k=n-k case} using induction and the local ``exchange of roots" of Ginzburg, Rallis and Soudry \cite{GRS5}. For the global analog,
one may utilize the global ``exchanging roots" (\cite{G,GRS3,Soudry5,RGS}). The result for an arbitrary maximal parabolic is given in
Theorem~\ref{theorem:twisted Jacquet modules of theta along maximal parabolic}.

Kable \cite{Kable} computed the nontwisted Jacquet modules, i.e.,
with the trivial character $\psi_0$, and expressed them as metaplectic tensors. This already had several applications \cite{Kable,Kable2,me9}.

The particular case of $n=2k$ and the generic character $\psi_k$ implies a notion of a ``metaplectic Shalika model" ($\psi_k$ takes the form of the usual Shalika character). In this case, Theorem~\ref{theorem:intro Jacquet modules} implies the existence of a unique (up to a scalar) nontrivial linear functional on the space of $\theta$, which is $\psi_k$-equivariant on the left under $U_{k,k}$. In the non-metaplectic setting, for the existence of a Shalika
functional, the stabilizer $\mathbb{GL}_k$ must then act trivially. The correspondence between representations of $\mathbb{GL}_k$ and $\widetilde{\mathbb{GL}}_k$ (recall that this cover is trivial) takes
$1$ to $1\otimes\gamma_{\psi}$, which is essentially the representation $\gamma_{\psi,(-1)^{j-1}}$ we obtained. In turn, we deduce a metaplectic Shalika model for $\theta$, which we use for the proof of Theorem~\ref{theorem:tau otimes tau dual is dist}. For more details see
Remark~\ref{remark:about metaplectic Shalika model}. We mention that this model may have applications similar to those of \cite{FJ,JR2}.

Let $\GSpin_{2n+1}$ be the split odd general spin group of rank $n+1$.
A result similar to Theorem~\ref{theorem:intro Jacquet modules}, for exceptional
representations $\Theta$ of $\GSpin_{2n+1}$ (defined in \cite{me8} following the exposition in \cite{BFG} for $\SO_{2n+1}$), was proved in \cite{me10} and used in a study of $\Theta\otimes\Theta'$. The definition of distinguished representations of $\GSpin_{2n+1}$ is similar to that of $\GL_n$,
an irreducible representation is distinguished if its contragradient is a quotient of $\Theta\otimes\Theta'$. The results of
\cite{me10} allow us to alternate between quotients of $\theta\otimes\theta'$ and $\Theta\otimes\Theta'$, using parabolic induction.

One essential difference between these settings, is that in contrast with $\theta\otimes\theta'$, the space
$\Theta\otimes\Theta'$ does not afford a Whittaker functional (\cite{me10} Theorem~1).
This interplay will be used
as an ingredient in a forthcoming work on the conjecture of Lapid and Mao on
Whittaker-Fourier coefficients, for even orthogonal groups (\cite{LM2}).
One can also use this to deduce certain irreducibility results (see Remark~\ref{remark:irred}).
In more detail, Proposition~4.1 of \cite{me10} showed that for a tempered distinguished $\tau$,
a representation parabolically induced from $\nu^{1/2}\tau\otimes1$ to $\GSpin_{2n+1}$ will also be distinguished, but then it must be reducible,
otherwise it is a generic quotient of $\Theta\otimes\Theta'$.

A result similar to a global formulation of Theorem~\ref{theorem:intro Jacquet modules} was used to compute a global co-period integral,
involving the integration of the residue of an Eisenstein series, against a pair of automorphic forms in the spaces of the small representation
of $\SO_{2n+1}$ or $\GSpin_{2n+1}$ \cite{me7,me8}.

To prove that a square-integrable representation $\tau$ such that $L(s,\tau,\mathrm{Sym}^2)$ has a pole at $s=0$, is distinguished, we apply
the following globalization lemma. As mentioned above, in one case its proof relies on an assumption, explicitly given in Remark~\ref{remark:exact assumptions for globalization lemma}.
\begin{theo}\label{theorem:intro globalization}
Let $\pi$ be a square-integrable representation of $\GL_n$. Assume that
$L(s,\pi,\mathrm{R})$ has a pole at $s=0$, for $\mathrm{R}=\mathrm{Sym}^2$ or $\bigwedge^2$. Then there exist
a number field with a ring of ad\`{e}les $\Adele$ and a global cuspidal representation
$\Pi$ of $\GL_n(\Adele)$, such that for a sufficiently large finite set of places $S$, $L^S(s,\Pi,\mathrm{R})$ has a pole at $s=1$, and at some place $v$, $\Pi_{v}=\pi$.
\end{theo}
See Lemma~\ref{lemma:globalization lemma}. This result is expected to have other applications. 

There is a minor ``historical gap" in the theory of exceptional representations, regarding their Whittaker models.
Let $F$ be any local non-Archimedean field of characteristic different from $2$.
Kazhdan and Patterson \cite{KP} proved that for $n\geq3$, if $|2|\ne1$ in $F$, the
exceptional representations do not have Whittaker models. For $n=3$, Flicker, Kazhdan and Savin \cite{FKS} (Lemma~6) used global methods to
extend this result to the case $|2|=1$. The remaining case ($|2|=1$ and $n>3$) has been expected, but not proved. We complete the proof in
Theorem~\ref{theorem:theta non generic n greater than 3}. This immediately validates several results from \cite{BG,Kable,Kable2,me9} when the field
is dyadic, including the local functional equation of \cite{BG} (Section~5) and the aforementioned conjecture of Savin (see \cite{Kable2} p.~1602),
which now hold in general.



Exceptional representations for $\widetilde{\GL}_n$ were introduced and studied by Kazhdan and Patterson \cite{KP}. They are related to a broader class of small, or minimal, representations. The first example was probably the Weil representation of $\widetilde{\Sp}_{n}$.
These representations enjoy the vanishing of a large class of Fourier coefficients, which makes them valuable for applications involving lifts and Rankin-Selberg integrals \cite{GRS6,G2,GJS2}. They have had a profound role in the theory of representations and appeared in numerous studies, including
\cite{V,KZ,KS,Savin2,BK,Savin,GRS,BFG3,GRS3,KP3,BFG,JSd1,GanSavin,Soudry4,LokeSavin,RGS}.

To date, the significant application of the exceptional representations of Kazhdan and Patterson \cite{KP}, is the construction of a Rankin-Selberg integral representation for the symmetric square $L$-function, by Bump and Ginzburg \cite{BG}.
Takeda \cite{Tk} extended the results of \cite{BG}, to some extent, to the twisted symmetric square $L$-function.

The term ``distinguished" has been used in various contexts. Let $\xi$ be a representation of a group $G$ and let $\eta$ be a character of a subgroup $H<G$. The representation $\xi$ is called $(H,\eta)$-distinguished, if $\Hom_H(\xi,\eta)\ne0$.
There are numerous studies on local and global distinguished representations, including \cite{Jac2,JR,FJ,Offen,OS2,OS,Offen2,Jac3,Matringe3,Matringe2,Matringe,FLO,Matringe5}.

Let $F_0$ be a quadratic extension of $F$.
Matringe \cite{Matringe3,Matringe2,Matringe} studied $(\GL_n(F),\eta)$-distinguished representations of $\GL_n(F_0)$.
He proved, using different (but in some sense related) techniques, a combinatorial characterization similar to Theorem~\ref{theorem:intro characterization}. He also proved (\cite{Matringe2}) that an irreducible generic representation $\xi$ is distinguished, if and only if its Rankin-Selberg Asai $L$-function has an exceptional pole at $0$. He then showed that $L(s,\xi,\mathrm{Asai})=L(s,\rho(\xi),\mathrm{Asai})$, where $\rho(\xi)$ is the Langlands parameter associated with $\xi$ (\cite{Matringe4}).

Feigon, Lapid and Offen \cite{FLO} studied (local and global) representations distinguished by unitary groups.


The rest of this work is organized as follows. Section~\ref{section:preliminaries} contains the preliminaries, including a brief
review of the metaplectic tensor of Kable \cite{Kable} and exceptional representations. The Jacquet modules are computed
in Section~\ref{section:Heisenberg jacquet modules}. Our results on distinguished representations occupy Section~\ref{section:Distinguished representations}.

\subsection*{Acknowledgments}
I wish to express my gratitude to Erez Lapid for suggesting this project to me and for numerous valuable, encouraging and inspiring discussions.
I thank David Soudry for communicating the results of Liu \cite{BL} and Jantzen and Liu \cite{JB} to me. I thank Baiying Liu for helpful conversations.
I would like to thank Chris Jantzen for his help with Claim~\ref{claim:quasi split lift}. 
Lastly, I thank Jim Cogdell for his kind encouragement and useful remarks.

\section{Preliminaries}\label{section:preliminaries}

\subsection{The groups}\label{subsection:the groups}
Let $F$ be a local \nonarchimedean\ field of characteristic different from $2$. 
Let $(,)_2$ be the Hilbert symbol of order $2$ of $F$ and put $\mu_2=\{-1,1\}$.
For a group $G$, $C_G$ denotes its center. If $x,y\in G$ and $Y<G$, $\rconj{x}y=xyx^{-1}$,
$\rconj{x}Y=\setof{\rconj{x}y}{y\in Y}$. Also if $d\in\mathbb{Z}$, $Y^d=\{y^d:y\in Y\}$ and in particular $F^{*d}=(F^*)^d$. The set of $m\times k$ matrices over $F$ is denoted $F^{m\times k}$.

Note that our results up to and including Corollary~\ref{corollary:correctness for irred generic} apply to any
local $p$-adic field of characteristic different from $2$. Then we add the assumption that the
characteristic of $F$ is $0$ (see also Remark~\ref{remark:extending combinatorial theorem for any field}).
Henceforth we omit references to the field, e.g., $\GL_n=\GL_n(F)$.

Fix the Borel subgroup of upper triangular invertible matrices $B_{n}=T_{n}\ltimes N_{n}$, where $T_{n}$ is the diagonal torus.
A partition of $n$ into $m$ parts is an $m$-tuple $\alpha=(\alpha_1,\ldots,\alpha_m)$ of nonnegative integers, whose sum is $n$
(partitions will be regarded as ordered partitions). For a
partition $\alpha$, let $Q_{\alpha}=M_{\alpha}\ltimes U_{\alpha}$ denote the corresponding standard parabolic subgroup; its Levi
part $M_{\alpha}$ is isomorphic to $\GL_{\alpha_1}\times\ldots\times\GL_{\alpha_m}$. In particular the maximal parabolic subgroups
are the subgroups $Q_{n-k,k}$ for $0\leq k\leq n$; $Q_n=Q_{0,n}=Q_{n,0}=\GL_n$. 
The group $\GL_k$ is regarded as a subgroup of $\GL_n$ via its natural embedding in $M_{k,n-k}$.
For any parabolic subgroup $Q<\GL_n$, let $\delta_{Q}$ denote its modulus character.

Let $\widetilde{\GL}_n$ be the metaplectic double cover of $\GL_n$, as constructed by
Kazhdan and Patterson \cite{KP} (with their $c$ parameter equal to $0$). They defined their cover by restriction from the cover of
$\mathrm{SL}_{n+1}$ of Matsumoto \cite{Mats}. We use the block-compatible cocycle $\sigma=\sigma_n$ of Banks, Levi and Sepanski \cite{BLS}, which coincides with the cocycle of Kubota \cite{Kubota} for $n=2$. The block-compatibility property reads, for $a,a'\in\GL_{n-k}$ and $b,b'\in\GL_{k}$,
\begin{align}\label{eq:block-compatibility}
\sigma(\mathrm{diag}(a,b),\mathrm{diag}(a',b'))=\sigma_{n-k}(a,a')\sigma_{k}(b,b')(\det{a},\det{b'})_2.
\end{align}

Let $\mathfrak{s}:\GL_n\rightarrow\widetilde{GL}_n$ be the section of \cite{BLS} and
$p:\widetilde{\GL}_n\rightarrow\GL_n$ be the natural projection.
For any subset $X\subset \GL_n(F)$, denote $\widetilde{X}=p^{-1}(X)$. We pull back the determinant to a non-genuine character
of $\widetilde{\GL}_n$, also denoted $\det$, and using this any character of $F^*$ can be pulled back to a character of $\widetilde{\GL}_n$.
Let $\mathe$ be $1$ if $n$ is odd, otherwise $\mathe=2$. Then $C_{\widetilde{\GL}_n}=p^{-1}(C_{\GL_n}^{\mathe})$.

\subsection{Representations}\label{subsection:representations}
Let $G$ be an $l$-group (\cite{BZ1} 1.1). Representations of $G$ will be complex and smooth.
Let $\Alg{G}$ denote the category of these representations and $\Alg_{\mathrm{irr}}{G}\subset\Alg{G}$ the subcategory of irreducible representations.
For $\pi\in \Alg{G}$, $\pi^{\vee}$ is the representation contragradient to $\pi$. If $\pi$ admits a central character, it will be
denoted $\omega_{\pi}$. If $H<G$ and $g\in G$, $\rconj{g}\pi$ is the representation of $\rconj{g}H$ on the space of $\pi$ given by
$\rconj{g}\pi(x)=\pi(\rconj{g^{-1}}x)$. We say that $\pi$ is glued from representations $\pi_1,\ldots,\pi_k$, if $\pi$ has a filtration, whose quotients (which may be isomorphic or zero) are, after a permutation, $\pi_1,\ldots,\pi_k$. 

Assume that $\widetilde{G}$ is a given central extension of $G$ by $\mu_2$ and $\varphi:G\rightarrow\widetilde{G}$ is a section. If $\pi$ and $\pi'$ are genuine representations of $\widetilde{G}$, $\pi\otimes \pi'$ (outer tensor product) is a representation of $G$ via $g\mapsto\pi(\varphi(g))\otimes\pi'(\varphi(g))$. The definition is independent of the choice of $\varphi$ and the actual section will be omitted.

Let $X$ be an $l$-space (\cite{BZ1} 1.1). The space of Schwartz-Bruhat functions on $X$ is denoted $\mathcal{S}(X)$. If $G$ acts on $X$,
it also acts on $\mathcal{S}(X)$.

Regular induction is denoted $\Ind$ and $\ind$ is the compact induction.
Induction is not normalized. In $\GL_n$, if $\alpha=(\alpha_1,\ldots,\alpha_m)$ is a partition of $n$ and $\tau_1\otimes\ldots\otimes\tau_m\in\Alg{M_{\alpha}}$, i.e.,
$\tau_i\in\Alg{\GL_{\alpha_i}}$, then $\tau_1\times\ldots\times\tau_m=\Ind_{Q_{\alpha}}^{\GL_n}(\delta_{Q_{\alpha}}^{1/2}\tau_1\otimes\ldots\otimes\tau_m)$.

Let $\Alg_{\mathrm{sqr}}{\GL_n}\subset\Alg_{\mathrm{esqr}}{\GL_n}\subset\Alg_{\mathrm{irr}}{\GL_n}$ be the subcategories of square-integrable and
essentially square-integrable representations. We briefly recall the characterization of essentially square-integrable representations of Zelevinsky \cite{Z3} (Section~9).
Let $\Alg_{\mathrm{cusp}}{\GL_n}\subset\Alg_{\mathrm{esqr}}{\GL_n}$ denote the subcategory of supercuspidal representations (not necessarily
unitary). For brevity, put $\nu=|\det{}|$. Let $\mathscr{C}$ be the set of equivalence classes of $\Alg_{\mathrm{cusp}}{\GL_n}$ for all $n\geq0$. A segment in $\mathscr{C}$
is a subset $[\rho,\nu^{l-1}\rho]=\{\rho,\nu\rho\ldots,\nu^{l-1}\rho\}$, where $\rho\in\Alg_{\mathrm{cusp}}{\GL_k}$ and $l\geq1$ is an integer. The representation
$\rho\times\nu\rho\times\ldots\times\nu^{l-1}\rho$
has a unique irreducible quotient denoted $\langle[\rho,\nu^{l-1}\rho]\rangle^t$. Then $\tau\in\Alg_{\mathrm{esqr}}{\GL_n}$ if and only if it is isomorphic
to some representation $\langle[\rho,\nu^{l-1}\rho]\rangle^t$. In this case $\omega_{\tau}=\nu^{(l-1)/2}\omega_{\rho}$. Also note that
$(\langle[\rho,\nu^{l-1}\rho]\rangle^t)^{\vee}=\langle[\nu^{1-l}\rho^{\vee},\rho^{\vee}]\rangle^t$.

Zelevinsky \cite{Z3} also described the reducibility of products of essentially square-integrable representations.
Two segments $[\rho,\nu^{l-1}\rho]$ and $[\rho',\nu^{l'-1}\rho']$ are called linked, if neither of them is a subset of the other,
but their union is a segment. If they are linked, in particular $\rho=\nu^a\rho'$ for some integer $a$. The representation
$\langle[\rho_1,\nu^{l_1-1}\rho_1]\rangle^t\times\ldots\times\langle[\rho_m,\nu^{l_m-1}\rho_m]\rangle^t$ is irreducible if and only if no pair of segments are linked.

Let $\psi$ be a nontrivial additive character of $F$.
Denote the normalized Weil factor by
$\gamma_{\psi}$ (\cite{We} Section~14), $\gamma_{\psi}(\cdot)^4=1$. 
For $a\in F^*$, let $\gamma_{\psi,a}$ be the Weil factor corresponding to the character $x\mapsto\psi(ax)$. 
Recall the following formulas (see the appendix of \cite{Rao}):
\begin{align}\label{eq:Weil factor identities}
\gamma_{\psi}(xy)=\gamma_{\psi}(x)\gamma_{\psi}(y)(x,y)_2 ,\qquad \gamma_{\psi}(x^2)=1,\qquad \gamma_{\psi,a}(x)=(a,x)_2\gamma_{\psi}(x).
\end{align}
We also denote by $\psi$ the generic character of $N_n$ given by $\psi(u)=\psi(\sum_{i=1}^{n-1}u_{i,i+1})$.
By restriction, it is a character of any unipotent subgroup $U<N_n$. When we define a character of such a subgroup, we usually use the notation
$\psi$, if it coincides with this restriction.

Let $\tau\in\Alg_{\mathrm{irr}}{\GL_{n}}$. A ($\psi$-)Whittaker functional on $\tau$ is
a functional $\lambda$ such that $\lambda(u\varphi)=\psi(u)\lambda(\varphi)$ for any $u\in N_n$ and $\varphi$ in the space of $\tau$.
We say that $\tau$ is generic, if it admits a nontrivial Whittaker functional.

If $n=2k$ we can also consider a Shalika functional, the equivariance property now reads
$\lambda(\left(\begin{smallmatrix}c&v\\&c\end{smallmatrix}\right)\varphi)=\psi(\mathrm{tr}(v))\lambda(\varphi)$ for any $c\in\GL_k$ and $v\in F^{k\times k}$. If $\lambda\ne0$, it yields an embedding $\tau\subset\Ind_{\mathbb{GL}_kU}^{\GL_{2k}}(\psi)$, where $\mathbb{GL}_k=\{\mathrm{diag}(c,c):c\in\GL_k\}$, called a Shalika model of $\tau$. If it exists, uniqueness of this model was proved by Jacquet and Rallis \cite{JR2} in the \nonarchimedean\ case and by Ash and Ginburg \cite{AG} over \archimedean\ fields.

\subsection{Jacquet modules}\label{subsection:Jacquet modules}
Let $\pi\in\Alg{G}$. Let $U<G$ be a closed subgroup, exhausted by its compact subgroups (here $U$ will be a unipotent subgroup of $\GL_n$) and $\psi$ be a character of $U$. Assume $M<G$ is the normalizer of $U$ and stabilizer of $\psi$. The Jacquet module of $\pi$ with
respect to $U$ and $\psi$ is denoted $\pi_{U,\psi}$. The action is not normalized. We have the following exact sequence in $\Alg{M}$,
\begin{align*}
0\rightarrow \pi(U,\psi)\rightarrow \pi\rightarrow \pi_{U,\psi}\rightarrow0.
\end{align*}
The representation $\pi(U,\psi)$ can be characterized by the Jacquet-Langlands lemma:
\begin{lemma}\label{lemma:Jacquet kernel as integral} (see e.g. \cite{BZ1} 2.33)
a vector $v$ in the space of $\pi$ belongs to $\pi(U,\psi)$ if and only if
\begin{align*}
\int_{\mathcal{U}}\pi(u)v\ \psi^{-1}(u)\ du=0,
\end{align*}
for some compact subgroup $\mathcal{U}<U$.
\end{lemma}
When $\psi=1$, we simply write $\pi(U)$ and $\pi_U$.
The following consequence of this lemma will be used.
\begin{lemma}\label{lemma:nontwisted Jacquet of tensor of two is one}
Let $Q$ be an $l$-group and $U<Q$ be a closed subgroup, exhausted by its compact subgroups, which is normal in $Q$.
The group $Q$ acts on the set of characters $\psi$ of $U$ by $\psi^g(u)=\psi(\rconj{g^{-1}}u)$, denote an orbit by $\mathcal{O}(\psi)$.
Assume $\psi$
and $\psi'$ are two characters of $U$ and let $L,L'<Q$ be their stabilizers. Assume $L,L',LU,L'U$ are closed subgroups of $Q$. Then for $\pi\in\Alg{LU}$ and $\pi'\in\Alg{L'U}$,
\begin{align*}
(\ind_{LU}^Q\pi_{U,\psi}\otimes\ind_{L'U}^Q\pi'_{U,\psi'})_U=
\begin{cases}\ind_{LU}^Q(\pi_{U,\psi}\otimes\pi'_{U,\psi^{-1}})&\psi'=\psi^{-1},\\
0&\psi'\not\in\mathcal{O}(\psi^{-1}).
\end{cases}
\end{align*}
\end{lemma}
\begin{proof}[Proof of Lemma~\ref{lemma:nontwisted Jacquet of tensor of two is one}]
Let $f$ belong to the space of $\ind_{LU}^Q\pi_{U,\psi}$. Choose a
small compact open neighborhood of the identity $\mathcal{V}<Q$ and a finite set $\Omega\subset Q$, such that $f$ is fixed by $\mathcal{V}$
and the support of $f$ is contained in $LU\Omega\mathcal{V}$. Similarly let $f'$ belong to the space of $\ind_{L'U}^Q\pi'_{U,\psi'}$ and choose
$\mathcal{V}'$ and $\Omega'$, analogously. For any compact $\mathcal{U}<U$ and $g,g'\in Q$,
\begin{align}\label{eq:nontwisted Jacqut of Jacquet first integral}
\int_{\mathcal{U}}u\cdot(f\otimes f')(g,g')\ du=
f(g)\otimes f'(g')\int_{\mathcal{U}}\psi(\rconj{g}u)\psi'(\rconj{g'}u)\ du.
\end{align}
This clearly vanishes unless $g=luxv$, where $l\in L$, $u\in U$, $x\in\Omega$ and $v\in\mathcal{V}$, and $g'=l'u'x'v'$ (with a similar notation).
In this case the integral equals
\begin{align}\label{eq:nontwisted Jacqut of Jacquet second integral}
\psi(u)\pi_{U,\psi}(l)f(x)\otimes \psi'(u)\pi'_{U,\psi'}(l')f'(x')\int_{\mathcal{U}}\psi(\rconj{xv}u)\psi'(\rconj{x'v'}u)\ du.
\end{align}

Assume $\psi'\not\in\mathcal{O}(\psi^{-1})$. Then for any $y,y'\in Q$,
$u\mapsto\psi(\rconj{y}u)\psi'(\rconj{y'}u)$ is a nontrivial character of $U$. Since $x$ and $x'$ vary in finite sets,
while $v$ and $v'$ vary in compact neighborhoods of the identity, one can choose a large enough compact $\mathcal{U}$ such that the
integral vanishes for all data $x,v,x',v'$. Thus Lemma~\ref{lemma:Jacquet kernel as integral} implies that $f\otimes f'$ vanishes in the Jacquet module.

Now assume $\psi'=\psi^{-1}$. Consider the mapping
\begin{align*}
f\otimes f'\mapsto (f\otimes f')(g)=f(g)\otimes f'(g).
\end{align*}
It is clearly onto $\ind_{LU}^Q(\pi_{U,\psi}\otimes\pi'_{U,\psi^{-1}})$.
If $f\otimes f'=0$ in the Jacquet module, by lemma~\ref{lemma:Jacquet kernel as integral} the \lhs\ of \eqref{eq:nontwisted Jacqut of Jacquet first integral} vanishes for all $g=g'\in Q$,
but then the $du$-integral on the \rhs\ becomes a nonzero constant. This implies $(f\otimes f')(g)=f(g)\otimes f'(g)=0$.

In the other direction assume
$f(g)\otimes f'(g)=0$ for all $g\in Q$. We must show that
\eqref{eq:nontwisted Jacqut of Jacquet first integral} vanishes for all $g,g'\in Q$. It clearly does if $g=g'$. Take $g\ne g'$. Using the notation above write $g=luxv$ and $g'=l'u'x'v'$ and consider the \rhs\ of \eqref{eq:nontwisted Jacqut of Jacquet first integral}. 

If the character $u\mapsto \psi(\rconj{g}u)\psi^{-1}(\rconj{g'}u)$ is trivial on $U$, $g'g^{-1}=l''u''\in LU$, but then
\begin{align*}
f(g)\otimes f'(g')=
f(g)\otimes\psi^{-1}(u'')\pi'_{U,\psi^{-1}}(l'')f'(g).
\end{align*}
By our assumption either $f(g)=0$ or $f'(g)=0$ and in the latter case, $\pi'_{U,\psi^{-1}}(l'')f'(g)=0$ for any $l''\in L$. Thus
$f(g)\otimes f'(g')=0$.

For a pair $g,g'$ such that this character is a nontrivial character,
as above using the fact that $x,x'$ vary in finite sets and $v,v'$ belong to compact subgroups, one can choose a large
enough $\mathcal{U}$ so that this integral vanishes.

Finally to pass from pure tensors to the general case, note that a general element of $\ind_{LU}^Q\pi_{U,\psi}\otimes\ind_{LU}^Q\pi'_{U,\psi^{-1}}$ can be written as a finite sum $\sum_i f_{i}\otimes f'_i$ with the following property:
for any subset of indices $i_1,\ldots,i_r$ such that the supports of $f_{i_j}\otimes f'_{i_j}$ coincide for $1\leq j\leq r$, for each $g,g'$ in the
support of $f_{i_1}\otimes f'_{i_1}$, the vectors
$(f_{i_1}\otimes f'_{i_1})(g,g'),\ldots,(f_{i_r}\otimes f'_{i_r})(g,g')$ are linearly independent in $\pi_{U,\psi}\otimes\pi'_{U,\psi^{-1}}$. Then
$(\sum_i f_{i}\otimes f'_i)(g,g')=0$ if and only if for each $i$, $(f_{i}\otimes f'_i)(g,g')=0$.
\end{proof}

\subsection{Metaplectic tensor}\label{subsection:The metaplectic tensor}
Let $\alpha$ be a partition of $n$. Irreducible representations of $M_{\alpha}$ can be described using the tensor product.
The situation in $\widetilde{M}_{\alpha}$ is more complicated, because the direct factors of $M_{\alpha}$ do not commute in the cover, then
the tensor construction breaks down. An alternative definition of a metaplectic tensor was developed by Kable \cite{Kable}, we briefly describe
his construction. Other studies of metaplectic tensor include \cite{FK,Su2,Mezo,Tk2}. The definition of Kable has several advantages, including the variety
of useful properties it enjoys and its relation to the exceptional representations (see
\eqref{eq:Jacquet nontwisted factors throuh metaplectic} and \eqref{eq:result of Kable for Jacquet module of exceptional} below).

For any closed subgroup $H<\GL_n$, Let $H^{\square}=\{h\in H:\det{h}\in F^{*2}\}$.
If $\pi\in\Alg{\widetilde{H}}$, denote its restriction to
$\widetilde{H}^{\square}$ by $\pi^{\square}$.

Let $\alpha_i$ be a partition of $n_i$ and $\pi_i\in\Alg_{\mathrm{irr}}{\widetilde{M}_{\alpha_i}}$ be genuine, $i=1,2$. Put
$n=n_1+n_2$ and let $\alpha=(\alpha_1,\alpha_2)$, i.e., $\alpha$ is the concatenation of $\alpha_1$ and $\alpha_2$. Then
\begin{align*}
(\widetilde{M}_{\alpha_1}^{\square},\widetilde{M}_{\alpha_2}^{\square}),\qquad(\widetilde{M}_{\alpha_1}^{\square},\widetilde{M}_{\alpha_2}),\qquad
(\widetilde{M}_{\alpha_1},\widetilde{M}_{\alpha_2}^{\square})
\end{align*}
are pairs of commuting subgroups (see \eqref{eq:block-compatibility}). The representation $\pi_1^{\square}\otimes\pi_2^{\square}$ of $p^{-1}(M_{\alpha_1}^{\square}\times M_{\alpha_2}^{\square})$ is genuine and similarly,
$\pi_1^{\square}\otimes\pi_2$ and $\pi_1\otimes\pi_2^{\square}$. Let $\omega$ be a genuine character of $C_{\widetilde{\GL}_n}$
which agrees with
\begin{align*}
\omega_{\pi_1}|_{p^{-1}(C_{\GL_{n_1}}^2)}\otimes\omega_{\pi_2}|_{p^{-1}(C_{\GL_{n_2}}^2)}
\end{align*}
on $p^{-1}(C_{\GL_{n}}^2)$. The metaplectic tensor $\pi_1\widetilde{\otimes}_{\omega}\pi_2$ was defined by Kable \cite{Kable} as an irreducible
summand of
\begin{align}\label{space:induced space of metaplectic tensor}
\ind_{p^{-1}(M_{\alpha_1}^{\square}\times M_{\alpha_2}^{\square})}^{\widetilde{M}_{\alpha}}(\pi_1^{\square}\otimes\pi_2^{\square}),
\end{align}
on which $C_{\widetilde{\GL}_{n}}$ acts by $\omega$. 
Note that the definition in \cite{Kable} was more general and included genuine
admissible finite length representations, which admit a central character. Here we will only encounter the case of irreducible representations.

A more concrete description was given in \cite{Kable} (Corollary~3.1): if $n_2$ is even or $n_1$ and $n_2$ are odd, there is an irreducible
summand $\sigma\subset\pi_2^{\square}$ such that
\begin{align}\label{eq:refinement for tensor}
\pi_1\widetilde{\otimes}_{\omega}\pi_2=
\ind_{p^{-1}(M_{\alpha_1}\times M_{\alpha_2}^{\square})}^{\widetilde{M}_{\alpha}}(\pi_1\otimes\sigma).
\end{align}
If both $n_2$ and $n_1$ are even, $\sigma$ is arbitrary; if $n_2$ is even and $n_1$ is odd, $\sigma$ is uniquely determined by
the condition $\omega=\omega_{\pi_1}\otimes\omega_{\sigma}$ on $C_{\widetilde{\GL}_n}$; when both are odd, $\sigma=\pi_2^{\square}$.
Otherwise $n_2$ is odd and $n_1$ is even, the definition is similar with the roles of $n_1$ and $n_2$ reversed.

The metaplectic tensor is associative (\cite{Kable} Proposition~3.5) and if $\beta_i$ is a subpartition of $\alpha_i$, i.e.,
$\beta_i$ is a partition of $n_i$ with $Q_{\beta_i}<Q_{\alpha_i}$,
\begin{align}\label{eq:Jacquet nontwisted factors throuh metaplectic}
(\pi_1\widetilde{\otimes}_{\omega}\pi_2)_{U_{\beta_1}\times U_{\beta_2}}=
(\pi_1)_{U_{\beta_1}}\widetilde{\otimes}_{\omega}(\pi_2)_{U_{\beta_2}}
\end{align}
(\cite{Kable} Proposition~4.1). 

Kable did not consider the behavior of the metaplectic tensor with respect to twisted Jacquet functors, i.e., with a nontrivial character $\psi$. Such modules are not representations of standard Levi factors.
One may attempt to extend the definition, to some extent, to include
these cases. For our purposes it will suffice to use restriction, in order to compute twisted Jacquet functors applied to the metaplectic tensor.

\begin{claim}\label{claim:Mackey theory applied to metaplectic tensor}
Let $\alpha_i$ be a partition of $n_i$, $\pi_i\in\Alg_{\mathrm{irr}}{\widetilde{M}_{\alpha_i}}$ be genuine, $i=1,2$. Then
\begin{align*}
(\pi_1\widetilde{\otimes}_{\omega}\pi_2)|_{p^{-1}(M_{\alpha_1}^{\square}\times M_{\alpha_2})}=\begin{dcases}\pi_1^{\square}\otimes\pi_2
&\text{even $n_2$,}\\
\pi_1^{\square}\otimes\bigoplus_{g\in\lmodulo{\widetilde{M}_{\alpha_2}^{\square}}{\widetilde{M}_{\alpha_2}}}\chi_g\pi_2
&\text{odd $n_1$ and $n_2$,}\\
\bigoplus_{g\in\lmodulo{\widetilde{M}_{\alpha_1}^{\square}}{\widetilde{M}_{\alpha_1}}}{\rconj{g}\sigma}\otimes\chi_g\pi_2&\text{even $n_1$, odd $n_2$.}
\end{dcases}
\end{align*}
Here $\chi_g$ is the character of $\lmodulo{\widetilde{M}_{\alpha_i}^{\square}}{\widetilde{M}_{\alpha_i}}$ given by $\chi_g(x)=(\det{x},\det{g})_2$ and
$\sigma$ is an irreducible summand of $\pi_1^{\square}$.
Similar results hold mutatis mutandis when restricting to $p^{-1}(M_{\alpha_1}\times M_{\alpha_2}^{\square})$. Consequently 
\begin{align*}
(\pi_1\widetilde{\otimes}_{\omega}\pi_2)_{p^{-1}(M_{\alpha_1}^{\square}\times M_{\alpha_2}^{\square})}=
\begin{dcases}
[F^*:F^{*2}]\pi_1^{\square}\otimes\pi_2^{\square}&\text{$n_1$ and $n_2$ are odd,}\\
\pi_1^{\square}\otimes\pi_2^{\square}&\text{otherwise.}
\end{dcases}
\end{align*}
\end{claim}
\begin{proof}[Proof of Claim~\ref{claim:Mackey theory applied to metaplectic tensor}]
The assertions follow from \eqref{eq:refinement for tensor} by Mackey's theory, using \cite{Kable} (Propositions~3.1 and 3.2) and \eqref{eq:block-compatibility}. For details see \cite{me9}. Note that the restriction to $p^{-1}(M_{\alpha_1}^{\square}\times M_{\alpha_2}^{\square})$ was already computed in \cite{Kable} (Theorem~3.1).
\end{proof}

This claim has the following consequence, which will be used repeatedly.
\begin{corollary}\label{corollary:hom of tensor separates}
Let $\alpha_i$ be a partition of $n_i$, $\pi_i,\pi'_i\in\Alg_{\mathrm{irr}}{\widetilde{M}_{\alpha_i}}$ be genuine for $i=1,2$, and $\alpha=(\alpha_1,\alpha_2)$. Assume
that for an admissible representation $\tau_1\otimes\tau_2\in\Alg{M_{\alpha}}$,
\begin{align*}
\Hom_{M_{\alpha}}((\pi_1\widetilde{\otimes}_{\omega}\pi_2)\otimes(\pi_1'\widetilde{\otimes}_{\omega'}\pi'_2),\tau_1\otimes\tau_2)\ne0.
\end{align*}
Then for each $i$ there is some square-trivial character $\chi_i$ of $F^*$, such that
\begin{align*}
\Hom_{M_{\alpha_i}}(\pi_i\otimes\pi'_i,\chi_i\tau_i)\ne0.
\end{align*}
Moreover, if $n_i$ is even, $\chi_i$ is in fact trivial.
\end{corollary}
\begin{proof}[Proof of Corollary~\ref{corollary:hom of tensor separates}]
According to Claim~\ref{claim:Mackey theory applied to metaplectic tensor} and with the same notation, $(\pi_1\widetilde{\otimes}_{\omega}\pi_2)|_{p^{-1}(M_{\alpha_1}^{\square}\times M_{\alpha_2})}$ is
a finite direct sum of representations $\rconj{g}\sigma\otimes\chi_g\pi_2$. Note that this form includes all the possibilities listed in the claim
(e.g., $\sigma$ could be $\pi_1^{\square}$), and if $n_2$ is even, $\chi_g=1$. We use a similar notation
$\rconj{g'}\sigma'\otimes\chi_{g'}\pi'_2$ for this restriction of $\pi'_1\widetilde{\otimes}_{\omega'}\pi'_2$.
Hence for some $g$ and $g'$,
\begin{align*}
\Hom_{M_{\alpha_1}^{\square}\times M_{\alpha_2}}((\rconj{g}\sigma\otimes\rconj{g'}\sigma')\otimes(\chi_g\pi_2\otimes\chi_{g'}\pi'_2),\tau_1\otimes\tau_2)\ne0.
\end{align*}
Thus if $\chi_2^{-1}=\chi_g\chi_{g'}$, $\Hom_{M_{\alpha_2}}(\pi_2\otimes\pi'_2,\chi_2\tau_2)\ne0$ as claimed (here we used the admissibility of $\tau_1$ and $\tau_2$). If $n_2$ is even, $\chi_g=\chi_{g'}=1$ hence $\chi_2=1$. The other case is symmetric.
\end{proof}

\subsection{Exceptional representations}\label{subsection:The exceptional representations}
The exceptional representations were introduced and studied by Kazhdan and Patterson \cite{KP}.
Let $\xi$ be an exceptional character, that is, $\xi$ is a genuine character
of $C_{\widetilde{T}_n}=p^{-1}(T_n^2)C_{\widetilde{GL}_n}$, such that
\begin{align*}
\xi(\mathfrak{s}(\mathrm{diag}(I_{i-1},x^2,x^{-2},I_{n-i-1})))=|x|,\qquad\forall 1\leq i\leq n-1,x\in F^*.
\end{align*}
Let $\rho(\xi)$ be the corresponding genuine irreducible representation of $\widetilde{T}_n$ (for $n>1$,
$\widetilde{T}_n$ is a $2$-step nilpotent subgroup). The principal series representation $\Ind_{\widetilde{B}_{n}}^{\widetilde{GL}_n}(\delta_{B_n}^{1/2}\rho(\chi))$
has a unique irreducible quotient $\theta$, called an exceptional representation.

The exceptional characters $\xi$ are parameterized in the following manner. Let $\chi$ be a character of $F^*$. Let $\gamma:F^*\rightarrow\mathbb{C}^*$ be a mapping such that $\gamma(xy)=\gamma(x)\gamma(y)(x,y)_2^{\lfloor n/2\rfloor}$ and $\gamma(x^2)=1$ for all $x,y\in F^*$. We call such a mapping a pseudo-character. Note that the definition of a pseudo-character depends on $n$, to explicate this we will occasionally call it an $n$-pseudo-character. Define
\begin{align*}
\xi_{\chi,\gamma}(\epsilon\mathfrak{s}(zI_n)\mathfrak{s}(t))=\epsilon\gamma(z)\chi(z^n\det{t})\delta_{B_n}^{1/4}(t),\qquad \epsilon\in\mu_2,t\in T_n^2, z\in F^{*\mathe}.
\end{align*}
When $n$ is even, the choice of $\gamma$ is irrelevant. When $n\equiv1\ (4)$, $\gamma$ is a square-trivial character of $F^*$. If
$n\equiv3\ (4)$, $\gamma=\gamma_{\psi}$ for some nontrivial additive character $\psi$ of $F$. Note that $\sigma(zI_n,z'I_n)=(z,z')_2^{\lfloor n/2\rfloor}$.
The corresponding exceptional
representation will be denoted $\theta_{n,\chi,\gamma}$, or $\theta_{\chi,\gamma}$ when $n$ is clear from the context. Furthermore, sometimes we
simply denote an exceptional representation by $\theta_{n}$ or $\theta$, when the data $\chi$ and $\gamma$ do not affect the validity of the argument.
Note that $\omega_{\theta_{\chi,\gamma}}(\mathfrak{s}(zI_n))=\chi(z)^n\gamma(z)$.

The following simple claim was proved in \cite{me10} (in the proof of Claim~4.1):
\begin{claim}\label{claim:taking characters out of exceptional}
We have $\theta_{\chi,\gamma}=\chi\theta_{1,\gamma}$.
Additionally, if $\gamma_0$ is another pseudo-character,
$\theta_{\chi,\gamma}=\eta\theta_{\chi,\gamma_0}$ for some square-trivial character $\eta$ of $F^*$.
\end{claim}
\begin{proof}[Proof of Claim~\ref{claim:taking characters out of exceptional}]
The proof follows easily from the fact that $\gamma/\gamma_0$ is a square-trivial character of $F^*$.
\end{proof}

The Jacquet functor carries exceptional representations into exceptional representations of Levi subgroups. The following result is due to Kable \cite{Kable} (Theorem~5.1),
\begin{align}\label{eq:result of Kable for Jacquet module of exceptional}
(\theta_{n_1+n_2,\chi,\gamma})_{U_{n_1,n_2}}=\delta_{Q_{n_1,n_2}}^{1/4}\theta_{n_1,\chi,\gamma_1}\widetilde{\otimes}_{\chi\gamma}\theta_{n_2,\chi,\gamma_2}.
\end{align}
Here $\gamma_1$ and $\gamma_2$ are arbitrary (in \cite{Kable} this appears with $\chi=1$ and with the normalized Jacquet functor).
On the \rhs\ $\chi\gamma$ is regarded as the character $\epsilon\mathfrak{s}(zI_n)\mapsto\epsilon\chi(z)^n\gamma(z)$.

Kazhdan and Patterson \cite{KP} (Section~I.3, see also \cite{BG} p.~145 and \cite{Kable} Theorem~5.4) proved that for $n\geq3$, if $|2|\ne1$ in $F$, the
exceptional representations do not have Whittaker models. For $n=3$, Flicker, Kazhdan and Savin \cite{FKS} (Lemma~6) used global methods to
extend this result to the case $|2|=1$. The following theorem completes the proof that $\theta$ is non-generic also when $|2|=1$ and $n>3$.
\begin{theorem}\label{theorem:theta non generic n greater than 3}
For any $n\geq3$, $\theta=\theta_n$ does not have a Whittaker model.
\end{theorem}
\begin{proof}[Proof of Theorem~\ref{theorem:theta non generic n greater than 3}]
According to \eqref{eq:result of Kable for Jacquet module of exceptional} (proved in \cite{Kable} without any assumption on the field),
$\theta\subset\Ind_{\widetilde{Q}_{3,n-3}}^{\widetilde{\GL}_n}(\delta_{Q_{3,n-3}}^{1/4}\theta_{3}\widetilde{\otimes}\theta_{n-3})$. Hence it is enough to
prove
\begin{align*}
(\Ind_{\widetilde{Q}_{3,n-3}}^{\widetilde{\GL}_n}(\delta_{Q_{3,n-3}}^{1/4}\theta_{3}\widetilde{\otimes}\theta_{n-3}))_{N_n,\psi}=0.
\end{align*}
By virtue of the Geometric Lemma of Bernstein and Zelevinsky (\cite{BZ2} Theorem~5.2), this representation is glued from certain
Jacquet modules of $\theta_{3}\widetilde{\otimes}\theta_{n-3}$. We show that these Jacquet modules vanish, this will complete the proof.
Let $\mathcal{W}$ be a subset of Weyl elements such that
$\GL_n=\coprod_{w\in\mathcal{W}}Q_{3,n-3}w^{-1}N_n$. When
$\psi|_{\rconj{w}U_{3,n-3}\cap N_n}\ne1$, the quotient corresponding to $w$ vanishes. This implies there is only one quotient, corresponding to $w_0=\left(\begin{smallmatrix}&I_{n-3}\\I_{3}\end{smallmatrix}\right)$, namely
\begin{align}\label{eq:space to clear}
\delta\cdot(\theta_{3}\widetilde{\otimes}\theta_{n-3})_{N_{3}\times N_{n-3},\psi}.
\end{align}
Here $\delta$ is some modulus character and $N_{3}\times N_{n-3}<M_{3,n-3}$. According to Claim~\ref{claim:Mackey theory applied to metaplectic tensor}, when we restrict $\theta_{3}\widetilde{\otimes}\theta_{n-3}$ to $\widetilde{\GL}_{3}^{\square}\times\widetilde{\GL}_{n-3}^{\square}$, we obtain a finite direct sum of representations
$\theta_{3}^{\square}\otimes\theta_{n-3}^{\square}$. Since the Jacquet functor with respect to $N_{3}$ and $\psi$ commutes with this tensor and the restriction, and $(\theta_{3})_{N_{3},\psi}=0$ by \cite{FKS} (Lemma~6), \eqref{eq:space to clear} is zero.
\end{proof}

\begin{lemma}\label{lemma:twisted theta U1 factors through theta U2}
Assume $n\geq3$, $\theta=\theta_n$ and let $\psi$ be the character of $U_{1,n-1}$ given by $\psi(u)=\psi(u_{1,2})$. Then
$\theta_{U_{1,n-1},\psi}$ is a quotient of $\theta_{U_{2,n-2}}$. Similarly, $\theta_{U_{n-1,1},\psi}$ is a quotient of $\theta_{U_{n-2,2}}$, where
$\psi(u)=\psi(u_{n-1,n})$.
\end{lemma}
\begin{proof}[Proof of Lemma~\ref{lemma:twisted theta U1 factors through theta U2}]
The proof closely resembles Proposition~4 of Bump, Friedberg and Ginzburg \cite{BFG2}. We prove only the first assertion, the second is symmetric.
For $1\leq j\leq n-1$, let $V_j=U_{j,n-j}\cap M_{j-1,n-j+1}$ (e.g., $V_1=U_{1,n-1}$ and $\prod_{j=1}^{n-1}V_j=N_n$). Let $\psi$ be the character of
$\prod_{i=1}^jV_i$ defined by $\psi(v)=\psi(\sum_{i=1}^jv_{i,i+1})$. Then $(\theta_{V_1,\psi})_{V_2,\psi}=\theta_{V_1V_2,\psi}$ and $\GL_{n-j-1}<M_{j+1,n-j-1}$ stabilizes $\psi$ and acts
transitively on the nontrivial characters of $V_{j+1}$. Hence it suffices to show
\begin{align*}
\theta_{V_1V_2,\psi}=0,
\end{align*}
as this would imply that the action of $U_{2,n-2}$ on $\theta_{U_{1,n-1},\psi}$ is trivial.
Suppose otherwise and let $j\geq2$ be maximal such that
\begin{align}\label{eq:suppose contradiction lemma factors through U2}
\theta_{\prod_{i=1}^jV_i,\psi}\ne0.
\end{align}
Since for $n\geq3$, $\theta$ does not have a Whittaker model, we must have $j\leq n-2$.
Then
\begin{align*}
\theta_{\prod_{i=1}^jV_i,\psi}=(\theta_{\prod_{i=1}^jV_i,\psi})_{V_{j+1}}=
(\theta_{U_{j+1,n-j-1}})_{N_{j+1},\psi}=\delta_{Q_{j+1,n-j-1}}^{1/4}(\theta_{j+1}\widetilde{\otimes}\theta_{n-j-1})_{N_{j+1},\psi},
\end{align*}
where the last equality follows from \eqref{eq:result of Kable for Jacquet module of exceptional}.
As in the proof of Theorem~\ref{theorem:theta non generic n greater than 3} above, we use
Claim~\ref{claim:Mackey theory applied to metaplectic tensor} and restrict $\theta_{j+1}\widetilde{\otimes}\theta_{n-j-1}$ to $\widetilde{\GL}_{j+1}^{\square}\times\widetilde{\GL}_{n-j-1}^{\square}$, and since $(\theta_{j+1})_{N_{j+1},\psi}=0$, we obtain $(\theta_{j+1}\widetilde{\otimes}\theta_{n-j-1})_{N_{j+1},\psi}=0$ contradicting \eqref{eq:suppose contradiction lemma factors through U2}.
\end{proof}

\section{Twisted Jacquet modules of exceptional representations}\label{section:Heisenberg jacquet modules}
In this section we study twisted Jacquet modules of $\theta=\theta_{n,\chi,\gamma}$. Let $U=U_{n-k,k}$, for
$0<k<n$. The Levi subgroup $M=M_{n-k,k}$ acts on the set of characters of $U$, with $\min(n-k,k)+1$ orbits. We choose representatives
\begin{align*}
\psi_j(u)=\psi(\sum_{i=1}^ju_{n-k-i+1,n-k+j-i+1}),\quad 0\leq j\leq \min(n-k,k).
\end{align*}
For example when $n=5$ and $k=3$, $\psi_1(u)=\psi(u_{2,3})$ and $\psi_2(u)=\psi(u_{2,4}+u_{1,3})$. When $n=2k$, $\psi_k\left(\begin{smallmatrix}I_k&v\\&I_k\end{smallmatrix}\right)=\psi(\mathrm{tr}(v))$.
The stabilizer of $\psi_j$ in $M$ is
\begin{align*}
\mathrm{St}_{n,k}(\psi_j)=\left\{\left(\begin{array}{cccc}b&v\\&c\\&&c&y\\&&&d\end{array}\right):b\in\GL_{n-k-j},c\in\GL_j,d\in\GL_{k-j}\right\}.
\end{align*}

Let $\mathbb{GL}_{j}$ denote the embedding of $\GL_j$ in $\mathrm{St}_{n,k}(\psi_j)$ via the coordinates of $c$.
The restriction of the cover to $\mathbb{GL}_{j}$ 
gives a trivial
double cover. In fact by \eqref{eq:block-compatibility} if $c,c'\in\GL_j$,
\begin{align*}
\sigma(\mathrm{diag}(I_{n-k-j},c,c,I_{k-j}),\mathrm{diag}(I_{n-k-j},c',c',I_{k-j}))=(\det{c},\det{c'})_2.
\end{align*}
Hence any genuine representation of $\widetilde{\mathbb{GL}}_j$ takes the form $\pi\otimes\gamma_{\psi}$, where 
\begin{align*}
\pi\otimes\gamma_{\psi}(\epsilon\mathfrak{s}(\mathrm{diag}(I_{n-k-j},c,c,I_{k-j})))=\epsilon\gamma_{\psi}(\det c)\pi(c),\qquad \epsilon\in\mu_2.
\end{align*}

The following theorem was proved by Gelbart and Piatetski-Shapiro \cite{GP} (Theorem~2.2) for the case $k=1$. We extend it to any $k$.
\begin{theorem}\label{theorem:twisted Jacquet modules of theta along generic character in k=n-k case}
Assume $n=2k$. Then $\theta_{U,\psi_k}$ is one-dimensional and in particular, irreducible. The action of $\widetilde{\mathbb{GL}}_k$ on this space
is given by the character
\begin{align*}
\epsilon\mathfrak{s}(\mathrm{diag}(c,c))\mapsto\epsilon\chi^2(c)\gamma_{\psi,(-1)^{k-1}}(\det c),\qquad c\in\GL_k.
\end{align*}
\end{theorem}
\begin{remark}\label{remark:about metaplectic Shalika model}
It follows that the space of functionals $\lambda$ on $\theta$ such that $\lambda(\theta(u)\varphi)=\psi_k(u)\lambda(\varphi)$, where $\varphi$
belongs to the space of $\theta$ and $u\in U$, is one-dimensional. In the non-metaplectic setting, if $\pi\in\Alg{\GL_n}$ and
$\pi_{U,\psi_k}$ is the trivial character, $\pi$ admits a Shalika functional (see Section~\ref{subsection:representations}).
Since $\chi^2$ corresponds to
$\chi^2\otimes\gamma_{\psi}$ (as explained above, because the cover is trivial), it is natural to call $\lambda$ a metaplectic twisted Shalika functional, or simply a metaplectic Shalika
functional when $\chi=1$. Of course it implies an
embedding $\theta\subset\Ind_{\mathbb{GL}_kU}^{\GL_{2k}}(\chi^2\gamma_{\psi,(-1)^{k-1}}\otimes\psi_k)$, which is a metaplectic twisted Shalika model.
We will use this observation in Section~\ref{section:Distinguished representations}. One can expect uniqueness of this model, as in the non-metaplectic
setting.
\end{remark}
\begin{proof}[Proof of Theorem~\ref{theorem:twisted Jacquet modules of theta along generic character in k=n-k case}]
We use induction on $k$. For $k=1$ the result is known: one-dimensionality holds because then $\theta$ is generic, and the action
of $\widetilde{\mathbb{GL}}_1$ is
$\epsilon\mathfrak{s}(\mathrm{diag}(c,c))\mapsto\epsilon\chi^2(c)\gamma_{\psi}(c)$ by \cite{GP} (Theorem~2.2, see also \cite{Kable} Lemma~5.3).

Denote the elements of $U$ by
\begin{align*}
v(v_1,v_2,v_3,v_4)=\left(\begin{array}{cccc}1&&v_1&v_2\\&I_{k-1}&v_3&v_4\\&&1\\&&&I_{k-1}\end{array}\right).
\end{align*}
By definition $\psi_k(v(v_1,v_2,v_3,v_4))=\psi(v_1)\psi_{k-1}(v_4)$. Let
\begin{align*}
&V_{3}=\{v(0,0,v_3,0)\}<V,\quad V_{1,2,4}=\{v(v_1,v_2,0,v_4)\}<V, 
\\&E=U_{1,k-1}=\left\{\left(\begin{array}{ccc}1&e\\&I_{k-1}\\&&I_{k}\end{array}\right)\right\}.
\end{align*}
According to the local analog of ``exchanging roots", proved by Ginzburg, Rallis and Soudry \cite{GRS5} (Lemma~2.2, stated for unipotent
subgroups of symplectic groups, but the arguments are general and hold in our setting), in $\Alg{V_{1,2,4}}$,
\begin{align}\label{iso:isomorphism exchange roots generic character n=2k}
\theta_{U,\psi_k}\isomorphic\theta_{V_{1,2,4}\rtimes E,\psi_k}.
\end{align}
Indeed, it is simple to check that the list of properties stated in the lemma holds in this setting (in the notation
of \cite{GRS5}, $C=V_{1,2,4}$, $X=V_3$ and $Y=E$).

Let
\begin{align*}
w=\left(\begin{array}{cccc}1\\&&\varepsilon\\&I_{k-1}\\&&&I_{k-1}\end{array}\right),\qquad\varepsilon=(-1)^{k-1}.
\end{align*}
Then $\rconj{\mathfrak{s}(w)}(\theta_{V_{2,3,4}\rtimes E,\psi_k})=\theta_{U',\psi'}$, where
\begin{align*}
U'=\left\{\left(\begin{array}{cccc}1&v_1&e&v_2\\&1&0&0\\&&I_{k-1}&v_4\\&&&I_{k-1}\end{array}\right)\right\},\qquad \psi'(u')=\psi(\varepsilon v_1)\psi_{k-1}(v_4).
\end{align*}
We chose $w\in\mathrm{SL}_{n}$, in order
to easily appeal to the formulas of Banks, Levi and Sepanski \cite{BLS} for computing conjugations of torus elements by $\mathfrak{s}(w)$ (see below). The actual choice of section $\mathfrak{s}$ does not matter for the conjugation, henceforth we simply write $w$.
Since $U'=U_{1,n-1}\rtimes U_{k-1,k-1}$, where $U_{k-1,k-1}<M_{2,2(k-1)}$, and $\psi'|_{U_{1,n-1}}=\psi(\varepsilon\cdot)$,
Lemma~\ref{lemma:twisted theta U1 factors through theta U2} implies
$\theta_{U',\psi'}=(\theta_{N_2U_{2,n-2},\psi^{\circ}})_{U_{k-1,k-1},\psi_{k-1}}$, where $\psi^{\circ}(u)=\psi(\varepsilon u_{1,2})$ and in particular,
$\psi^{\circ}|_{U_{2,n-2}}=1$. This Jacquet module
factors through the Jacquet module with respect to $U_{2,n-2}$ and the trivial character, hence by \eqref{eq:result of Kable for Jacquet module of exceptional}
it is equal to
\begin{align*}
\delta_{Q_{2,n-2}}^{1/4}(\theta_{2,\chi,\gamma_1}\widetilde{\otimes}_{\chi\gamma}\theta_{n-2,\chi,\gamma_2})_{N_2\times U_{k-1,k-1},\psi(\varepsilon\cdot)\otimes\psi_{k-1}}.
\end{align*}
Here $\gamma_1$ and $\gamma_2$ are arbitrary.
We restrict to
$\widetilde{\GL}_2^{\square}\times\widetilde{\GL}_{n-2}^{\square}$. Since $(\theta_2\widetilde{\otimes}\theta_{n-2})^{\square}=\theta_2^{\square}\otimes\theta_{n-2}^{\square}$ (Claim~\ref{claim:Mackey theory applied to metaplectic tensor}),
we obtain
\begin{align}\label{eq:applying factoring through}
\delta_{Q_{2,n-2}}^{1/4}((\theta_{2,\chi,\gamma_1})_{N_2,\psi(\varepsilon\cdot)})^{\square}\otimes((\theta_{n-2,\chi,\gamma_2})_{U_{k-1,k-1},\psi_{k-1}})^{\square}.
\end{align}
By the induction hypothesis both spaces are one-dimensional, hence $\theta_{U,\psi_k}$ is one-dimensional.

Regarding the action of $\widetilde{\mathbb{GL}}_k$, it is of the form
$\eta\otimes\gamma_{\psi}$ for some character $\eta$ of $F^*$, hence it is determined by its restriction to a maximal torus.

We describe the isomorphism \eqref{iso:isomorphism exchange roots generic character n=2k},
in order to understand how this action is transferred from $\theta_{U,\psi_k}$ to $\theta_{V_{1,2,4}\rtimes E,\psi_k}$.
The isomorphism was given in \cite{GRS5} (proof of Lemma~2.2) using the following chain of isomorphisms:
\begin{align*}
\theta_{V_{1,2,4}\rtimes E,\psi_k}\rightarrow
(\theta_{V_{1,2,4},\psi_k})_E\rightarrow
(\ind_{U}^{U\rtimes E}(\theta_{U,\psi_k}))_E\rightarrow
\theta_{U,\psi_k}.
\end{align*}
The first step was to define a mapping
$\theta_{V_{1,2,4},\psi_k}\rightarrow\ind_{U}^{U\rtimes E}(\theta_{U,\psi_k})$.
For $\varphi$ in the space of $\theta$,
$f(\varphi+\theta(V_{1,2,4},\psi_k))$ in the space of $\ind_{U}^{U\rtimes E}(\theta_{U,\psi_k})$ was given by
\begin{align*}
f(\varphi)(x)=\theta(x)\varphi + \theta(U,\psi_k),\qquad x\in U\rtimes E.
\end{align*}
This mapping was extended to a mapping $\theta_{V_{1,2,4}\rtimes E,\psi_k}\rightarrow(\ind_{U}^{U\rtimes E}(\theta_{U,\psi_k}))_E$ by
\begin{align*}
f(\varphi+\theta(V_{1,2,4}E,\psi_k))(x)=f(\varphi)(x)+(\ind_{U}^{U\rtimes E}(\theta_{U,\psi_k}))(E).
\end{align*}
To obtain an element in $\theta_{U,\psi_k}$, integrate $f(\varphi+\theta(V_{1,2,4}E,\psi_k))(x)$ over $E$. Altogether
\begin{align*}
\varphi+\theta(V_{1,2,4}E,\psi_k)\mapsto\int_{E}\theta(e)\varphi\ de+\theta(U,\psi_k).
\end{align*}
Let $\mathbb{T}_k$ denote the image of $T_k$ in $\mathbb{GL}_k$. Since $\mathbb{T}_k$ 
normalizes $V_{1,2,4}$, $E$ and $U$, and stabilizes $\psi_k$,
the isomorphism \eqref{iso:isomorphism exchange roots generic character n=2k} extends to $\Alg{\widetilde{\mathbb{T}}_k}$.

Let $t=\mathrm{diag}(t_1,t_2,t_1,t_2)\in T_n$ with $t_1\in F^*$ and $t_2\in T_{k-1}$ ($t$ is a general element of
$\mathbb{T}_k$). 
Since
\begin{align}\label{eq:explicating action of torus}
\theta(\mathfrak{s}(t))\varphi+\theta(V_{1,2,4}E,\psi_k)\mapsto\delta_{Q_{2,n-2}}^{1/4}(\rconj{w}t)\theta(\mathfrak{s}(t))\int_{E}\theta(e)\varphi\ de+\theta(U,\psi_k),
\end{align}
the action of $\widetilde{\mathbb{T}}_k$ on $\theta_{U,\psi_k}$ is transformed by \eqref{iso:isomorphism exchange roots generic character n=2k} to
$\rconj{w^{-1}}\delta_{Q_{2,n-2}}^{-1/4}$ multiplied by the action of $\widetilde{\mathbb{T}}_k$ on $\theta_{V_{2,3,4}\rtimes E,\psi_k}$. The latter is given by
$\theta_{V_{2,3,4}\rtimes E,\psi_k}(\mathfrak{s}(t))=\theta_{U',\psi'}(\rconj{w}\mathfrak{s}(t))$. Using \cite{BLS} (Section~2 Lemma~2, Section~3
Lemmas~3 and 1) we see that
\begin{align*}
\rconj{w}\mathfrak{s}(t)&=\sigma(w,t)\mathfrak{s}(\rconj{w}t)=(t_1,\det{t_2})_2(\det{t_2},\det{t_2})_2\mathfrak{s}(\mathrm{diag}(t_1,t_1,t_2,t_2))\\&=
(t_1,\det{t_2})_2(\det{t_2},\det{t_2})_2\mathfrak{s}(\mathrm{diag}(t_1,t_1,I_{2n-2}))
\mathfrak{s}(\mathrm{diag}(I_2,t_2,t_2))\in\widetilde{\GL}_2^{\square}\times\widetilde{\GL}_{n-2}^{\square}.
\end{align*}
According to the induction hypothesis
\begin{align*}
&(\theta_{2,\chi,\gamma_1})_{N_2,\psi(\varepsilon\cdot)}(\mathfrak{s}(\mathrm{diag}(t_1,t_1)))=\chi(t_1^2)\gamma_{\psi,\varepsilon}(t_1)=\chi(t_1^2)(t_1,t_1)_2^{k-1}\gamma_{\psi}(t_1),\\
&(\theta_{n-2,\chi,\gamma_2})_{U_{k-1,k-1},\psi_{k-1}}(\mathfrak{s}(\mathrm{diag}(t_2,t_2)))=\chi(\det{t_2}^2)\gamma_{\psi,-\varepsilon}(\det t_2)\\&\qquad=
\chi(\det{t_2}^2)(\det{t_2},\det{t_2})_2^{k-2}\gamma_{\psi}(\det t_2),
\end{align*}
therefore
\begin{align*}
&\theta_{U,\psi_k}(\mathfrak{s}(t))=\theta_{U',\psi'}(\rconj{w}\mathfrak{s}(t))=(\chi^2\gamma_{\psi,(-1)^{k-1}})(\det \mathrm{diag}(t_1,t_2)).
\end{align*}
Here we used \eqref{eq:Weil factor identities} and the fact that $(-1,a)_2=(a,a)_2$.
Note that the modulus character appearing in \eqref{eq:applying factoring through} was cancelled by the twist of the action due to
\eqref{eq:explicating action of torus}. 
\end{proof}
Using this result we can compute all the twisted Jacquet modules for maximal parabolic subgroups.
\begin{theorem}\label{theorem:twisted Jacquet modules of theta along maximal parabolic}
Let $0\leq j\leq\min(n-k,k)$. We have
\begin{align*}
\theta_{U,\psi_j}=\delta_{Q_{n-k-j,2j,k-j}}^{1/4}(\theta_{n-k-j,\chi,\gamma_1}\widetilde{\otimes}_{\chi\gamma_{(j)}}\theta_{k-j,\chi,\gamma_2})\otimes(\theta_{2j,\chi,\gamma_3})_{U_{j,j},\psi_j}.
\end{align*}
Here the metaplectic tensor is a representation of $p^{-1}(\GL_{n-k-j}\times\GL_{k-j})$, where $\GL_{n-k-j}\times\GL_{k-j}$ is embedded in $\GL_n$ through $\mathrm{St}_{n,k}(\psi_j)$; the pseudo-characters
$\gamma_i$, $1\leq i\leq 3$, are arbitrary; the pseudo-character $\gamma_{(j)}$ is arbitrary when $n$ is even, and uniquely determined by
\begin{align}\label{eq:condition to determine gamma j k}
\gamma_{(j)}(z)=\gamma(z)/\gamma_{\psi,(-1)^{j-1}}(z^j),\qquad\forall z\in F^*,
\end{align}
when $n$ is odd; the \rhs\ is regarded as a representation in $\Alg{\widetilde{\mathrm{St}}_{n,k}(\psi_j)}$ by extending it trivially on $U_{n-k-j,2j,k-j}$.
\end{theorem}
\begin{remark}
Note that $\gamma_{(j)}$ is really an $(n-2j)$-pseudo-character: when $j$ is even, $\gamma_{\psi,(-1)^{j-1}}(z^j)=1$ and $\lfloor (n-2j)/2\rfloor\equiv
\lfloor n/2\rfloor\ (2)$, while
if $j$ is odd, $\lfloor (n-2j)/2\rfloor\equiv\lfloor (n-2)/2\rfloor\ (2)$.
\end{remark}
\begin{proof}[Proof of Theorem~\ref{theorem:twisted Jacquet modules of theta along maximal parabolic}]
For $j=0$ this is \eqref{eq:result of Kable for Jacquet module of exceptional}, assume $j>0$. We can also assume that if $n=2k$, $j<k$, otherwise there is nothing to prove.
The assertion holds for $n=1,2$, assume $n\geq3$.
First we claim that $U_{n-k-j,2j,k-j}$ acts trivially on $\theta_{U,\psi_j}$.
\begin{claim}\label{claim:twisted Jacquet factors through larger unipotent}
$\theta_{U,\psi_j}=\theta_{U_{n-k-j,j,j,k-j},\psi_j}$.
\end{claim}
The proof is given below. Assuming this, according to \eqref{eq:result of Kable for Jacquet module of exceptional} and \eqref{eq:Jacquet nontwisted factors throuh metaplectic},
\begin{align*}
\theta_{U_{n-k-j,2j,k-j}}&=
\delta_{Q_{n-k-j,k+j}}^{1/4}
(\theta_{n-k-j,\chi,\gamma_1}\widetilde{\otimes}_{\chi\gamma}\theta_{k+j,\chi,\gamma'})_{U_{2j,k-j}}
\\&=\delta_{Q_{n-k-j,2j,k-j}}^{1/4}(\theta_{n-k-j,\chi,\gamma_1}\widetilde{\otimes}_{\chi\gamma}(\theta_{2j,\chi,\gamma_3}\widetilde{\otimes}_{\chi\gamma'}\theta_{k-j,\chi,\gamma_2})).
\end{align*}
Here $U_{2j,k-j}<\GL_{k+j}<M_{n-k-j,k+j}$. Hence, noting $\psi_j|_{U_{n-k-j,2j,k-j}}=1$,
\begin{align*}
\theta_{U,\psi_j}=(\theta_{U_{n-k-j,2j,k-j}})_{U_{j,j},\psi_j}=
\delta_{Q_{n-k-j,2j,k-j}}^{1/4}(\theta_{n-k-j,\chi,\gamma_1}\widetilde{\otimes}_{\chi\gamma}(\theta_{2j,\chi,\gamma_3}\widetilde{\otimes}_{\chi\gamma'}\theta_{k-j,\chi,\gamma_2}))_{U_{j,j},\psi_j},
\end{align*}
where $U_{j,j}<\GL_{2j}<M_{n-k-j,2j,k-j}$.

To compute the Jacquet functor with respect to $U_{j,j}$ and $\psi_j$, we use the Geometric Lemma of Bernstein and Zelevinsky \cite{BZ2} (Theorem~5.2).
According to the associativity of the metaplectic
tensor (\cite{Kable} Proposition~3.5), and the fact that induction of a semisimple representation from a subgroup of finite index is semisimple,
\begin{align}
\nonumber&\theta_{n-k-j,\chi,\gamma_1}\widetilde{\otimes}_{\chi\gamma}(\theta_{2j,\chi,\gamma_3}\widetilde{\otimes}_{\chi\gamma'}\theta_{k-j,\chi,\gamma_2})
\\\nonumber
&=\ind_{p^{-1}(\GL_{n-k-j}\times\GL_{2j}^{\square}\times\GL_{k-j})}^{\widetilde{M}_{n-k-j,2j,k-j}}((\theta_{n-k-j,\chi,\gamma_1}\widetilde{\otimes}_{\chi\gamma''}\theta_{k-j,\chi,\gamma_2})
\otimes\sigma),
\end{align}
where $\sigma\in\Alg{\widetilde{\GL}_{2j}^{\square}}$ is a suitable irreducible summand of $\theta_{2j,\chi,\gamma_3}^{\square}$, which depends on $\gamma''$
precisely when $n$ is odd. 
The double coset space $\rmodulo{\lmodulo{\GL_{2j}^{\square}}{\GL_{2j}}}{\mathbb{GL}_j}$ contains $[F^*:F^{*2}]$ representatives and we may take them in the form
$g_z=\mathrm{diag}(zI_j,I_j)$, where $z$ varies over a set of representatives of $\lmodulo{F^{*2}}{F^*}$. These representatives normalize
$\GL_{n-k-j}\times\GL_{2j}^{\square}\times\GL_{k-j}$, $\GL_{n-k-j}\times \mathbb{GL}_j\times\GL_{k-j}$ and
$U_{j,j}$, and $\rconj{g_z}\psi_{j}=\psi(z\cdot)_j$, i.e., $\rconj{g_z}\psi_{j}(\left(\begin{smallmatrix}I_j&v\\&I_j\end{smallmatrix}\right))=\psi(z\cdot\mathrm{tr}(v))$. 
Therefore
\begin{align*}
(\theta_{n-k-j,\chi,\gamma_1}\widetilde{\otimes}_{\chi\gamma}(\theta_{2j,\chi,\gamma_3}\widetilde{\otimes}_{\chi\gamma'}\theta_{k-j,\chi,\gamma_2}))_{U_{j,j},\psi_j}
\end{align*}
is glued from representations
\begin{align*}
\ind_{p^{-1}(\GL_{n-k-j}\times \mathbb{GL}_j^{\square}\times\GL_{k-j})}^{p^{-1}(\GL_{n-k-j}\times \mathbb{GL}_j\times\GL_{k-j})}(\rconj{g_z^{-1}}((\theta_{n-k-j,\chi,\gamma_1}\widetilde{\otimes}_{\chi\gamma''}\theta_{k-j,\chi,\gamma_2})\otimes
\sigma_{U_{j,j},\psi(z\cdot)_j})),\qquad z\in \lmodulo{F^{*2}}{F^*}.
\end{align*}
Since $\mathbb{GL}_j^{\square}=\mathbb{GL}_j$, $\widetilde{\mathbb{GL}}_j$ and $p^{-1}(\GL_{n-k-j}\times\GL_{k-j})$ commute 
whence each of these representations equals
\begin{align*}
\rconj{g_z^{-1}}(\theta_{n-k-j,\chi,\gamma_1}\widetilde{\otimes}_{\chi\gamma''}\theta_{k-j,\chi,\gamma_2})\otimes
\rconj{g_z^{-1}}\sigma_{U_{j,j},\psi(z\cdot)_j}.
\end{align*}
The block-compatibility formula \eqref{eq:block-compatibility} implies (see \cite{Kable} p.~748)
\begin{align*}
\rconj{g_z^{-1}}(\theta_{n-k-j,\chi,\gamma_1}\widetilde{\otimes}_{\chi\gamma''}\theta_{k-j,\chi,\gamma_2})=\chi_{z}\otimes(\theta_{n-k-j,\chi,\gamma_1}\widetilde{\otimes}_{\chi\gamma''}\theta_{k-j,\chi,\gamma_2}),
\end{align*}
where $\chi_z$ is the non-genuine character of $\widetilde{M}_{n-k-j,k-j}$ given by $\chi_z(g)=(\det{g},z^j)_2$. This is just a twist of the metaplectic
tensor and will eventually be reflected on the pseudo-character changing from $\gamma''$ to $\gamma_{(j)}$.
We show that $\rconj{g_z^{-1}}\sigma_{U_{j,j},\psi(z\cdot)_j}$ vanishes for all but one representative $z_0$, for which it equals
$(\theta_{2j,\chi,\gamma_3})_{U_{j,j},\psi_j}$ (both representations belong in $\Alg{\widetilde{\mathbb{GL}}_j}$). This will complete the proof.

In general let $\pi\in\Alg_{\mathrm{irr}}{\widetilde{\GL}_{2j}}$ be genuine. Then $\pi^{\square}$ is the direct sum of
$[F^*:F^{*2}]$ irreducible representations $\pi_i$ and moreover, for each fixed summand $\rho=\pi_i$, $\pi=\ind_{\widetilde{\GL}_{2j}^{\square}}^{\widetilde{\GL}_{2j}}(\rho)$ (\cite{Kable} Proposition~3.2). Now
an application of \cite{BZ2} (Theorem~5.2) similar to above implies that $\pi_{U_{j,j},\psi_j}$ is glued from representations
$\rconj{g_z^{-1}}\rho_{U_{j,j},\psi(z\cdot)_j}$, where $z$ varies over the different square classes of $F^*$. Furthermore, since $\psi$ and $\psi(y^2\cdot)$ are conjugates in $\widetilde{\GL}_{2j}^{\square}$ (use $g_{y^{-2}}$),
the spaces $\rho_{U_{j,j},\psi_j}$ and $\rho_{U_{j,j},\psi(y^2\cdot)_j}$ are isomorphic.
Assuming $\pi_{U_{j,j},\psi_j}$ is one-dimensional,
these observations imply that for each $\rho$, $\rconj{g_{z_0}^{-1}}\rho_{U_{j,j},\psi(z_0\cdot)_j}\ne0$
for some $z_0$ and furthermore, $z_0$ is unique modulo $F^{*2}$ with this property.

By Theorem~\ref{theorem:twisted Jacquet modules of theta along generic character in k=n-k case},
$(\theta_{2j,\chi,\gamma_3})_{U_{j,j},\psi_j}$ is one-dimensional, hence for each summand $\sigma$ there is a unique representative $z_0\in\lmodulo{F^{*2}}{F^*}$ such that $\rconj{g_{z_0}^{-1}}\sigma_{U_{j,j},\psi(z_0\cdot)_j}=(\theta_{2j,\chi,\gamma_3})_{U_{j,j},\psi_j}$.


\begin{proof}[Proof of Claim~\ref{claim:twisted Jacquet factors through larger unipotent}]
Let $V=U_{n-k-j,j}$ be embedded in $\mathrm{St}_{n,k}(\psi_j)$ as the subgroup of matrices corresponding to the coordinates of $v$
and denote its elements by
\begin{align*}
v(v_1,v_2,v_3,v_4)=\left(\begin{array}{cccc}I_{n-k-j-1}&&v_3&v_2\\&1&v_1&v_4\\&&1\\&&&I_{j-1}\end{array}\right).
\end{align*}
The group $V$ is abelian and normalizes $U$ (recall that $U=U_{n-k,k}$). First we show
\begin{align}\label{eq:top left radical acts trivially}
\theta_{U,\psi_j}=\theta_{UV,\psi_j}.
\end{align}

It is enough to prove
$(\theta_{U,\psi_j})_{V,\mu}=0$ for any nontrivial character $\mu$ of $V$. The group $\mathrm{St}_{n,k}(\psi_j)\cap M_{n-k-j,j,j}$ (i.e., the coordinates
of $b$ and $c$ in $\mathrm{St}_{n,k}(\psi_j)$) acts on the characters of $V$
and we may assume that $\mu$ does not depend on the coordinates of $v_3$ and $v_4$, and $\mu(v(v_1,0,0,0))=\psi(v_1)$.

As in the proof of Theorem~\ref{theorem:twisted Jacquet modules of theta along generic character in k=n-k case} and with a similar notation, if
\begin{align*}
V_{3}=\{v(0,0,v_3,0)\}<V,\quad V_{1,2,4}=\{v(v_1,v_2,0,v_4)\}<V, \quad E=\overline{U}_{n-k-j-1,1}<\GL_{n-k-j},
\end{align*}
where $\overline{U}_{n-k-j-1,1}$ is the unipotent radical opposite to $U_{n-k-j-1,1}$,
Lemma~2.2 of Ginzburg, Rallis and Soudry \cite{GRS5} implies the following isomorphism in $\Alg{V_{1,2,4}}$,
\begin{align*}
(\theta_{U,\psi_j})_{V,\mu}=(\theta_{U,\psi_j})_{V_{1,2,4}\rtimes E,\mu}.
\end{align*}

Conjugating the \rhs\ by $\left(\begin{smallmatrix}&I_{j+1}\\I_{n-k-j-1}\end{smallmatrix}\right)$,
we see that $(\theta_{U,\psi_j})_{V_{1,2,4}E,\mu}$ is a quotient of
\begin{align*}
(\theta_{U_{1,n-1},\psi})_{U_{2,n-2}\cap U,\psi'}, \qquad \psi(u)=\psi(u_{1,2}), \qquad \psi'(u)=\psi(u_{2,n-k+1}),
\end{align*}
where $\psi$ was obtained from $\mu|_{\{v(v_1,0,0,0)\}}$ and $\psi'$ from $\psi_j$. This space vanishes according to Lemma~\ref{lemma:twisted theta U1 factors through theta U2}.
Hence $(\theta_{U,\psi_j})_{V,\mu}=0$ and \eqref{eq:top left radical acts trivially} is proved.

Now let $Y=U_{j,k-j}$ be embedded in $\mathrm{St}_{n,k}(\psi_j)$ as the subgroup of matrices corresponding to the coordinates of $y$,
\begin{align*}
y=y(y_1,y_2,y_3,y_4)=\left(\begin{array}{cccc}I_{j-1}&&y_4&y_2\\&1&y_1&y_3\\&&1\\&&&I_{k-j-1}\end{array}\right).
\end{align*}
We show $(\theta_{UV,\psi_j})_{Y,\mu}=0$ for any nontrivial character $\mu$ of $Y$. The proof is a repetition of the argument above, this time
$\mu$ can be assumed to depend only on $y_1$ and $y_2$, we take $E=\overline{U_{1,k-j-1}}$ and conjugate using
$\left(\begin{smallmatrix}&I_{k-j-1}\\I_{j+1}\end{smallmatrix}\right)$ (embedded in the bottom right corner of $\GL_n$). Then the result follows
from Lemma~\ref{lemma:twisted theta U1 factors through theta U2} ($\theta_{U_{n-1,1},\psi}$ is a quotient of $\theta_{U_{n-2,2}}$).

Summing up, we have shown $(\theta_{U,\psi_j})=(\theta_{UVY,\psi_j})$ and since $UVY=U_{n-k-j,j,j,k-j}$, the claim is proved.
\end{proof}
\end{proof}

\section{Distinguished representations}\label{section:Distinguished representations}
Let $\tau\in\Alg{\GL_n}$ be admissible and assume that $\tau$ admits a central character $\omega_{\tau}$. Given a pair of exceptional representations
$\theta$ and $\theta'$ of $\widetilde{\GL}_n$, we say that $\tau$ is $(\theta,\theta')$-distinguished if
\begin{align}\label{homspace:main for distinguished}
\Hom_{\GL_n}(\theta\otimes\theta',\tau^{\vee})\ne0.
\end{align}
In light of Claim~\ref{claim:taking characters out of exceptional}, the spaces $\theta_{\chi,\gamma}$ as $\chi$ and $\gamma$ vary are
twists of each other by characters of $\GL_n$. Since $\theta_{\chi,\gamma}=\chi\theta_{1,\gamma}$, it is natural to
fix $\chi=\chi'=1$. Define $\tau$ to be distinguished, if it is $(\theta_{1,\gamma},\theta_{1,\gamma'})$-distinguished for some pair of pseudo-characters $\gamma$ and $\gamma'$.

A distinguished representation $\tau$ always satisfies $\omega_{\tau}^2=1$. In particular if $\tau\in\Alg_{\mathrm{esqr}}{\GL_n}$ is distinguished,
it is already unitary whence $\tau\in\Alg_{\mathrm{sqr}}{\GL_n}$. If $n=1$, these are the square-trivial characters of $F^*$.

The following simple claim explains the motivation for removing any specific choice of pseudo-characters from the definition.
Note that when $n$ is even, the pseudo-characters are redundant because they do not affect the exceptional representations (see Section~\ref{subsection:The exceptional representations}).
\begin{claim}\label{claim:eta tau is dist when tau is for odd n}
If $n$ is odd, for any square-trivial character $\eta$ of $F^*$, $\tau$ is distinguished if and only if $\eta\tau$ is.
\end{claim}
\begin{proof}[Proof of Claim~\ref{claim:eta tau is dist when tau is for odd n}]
For a fixed pseudo-character $\gamma_0$, the set $\gamma/\gamma_0$ with $\gamma$ varying over the
set of pseudo-characters, exhausts all square-trivial characters
of $F^*$. Then if $\tau$ is $(\theta_{1,\gamma},\theta_{1,\gamma_0})$-distinguished and $\eta=\gamma'/\gamma_0$,
$\eta\theta_{1,\gamma_0}=\theta_{1,\gamma'}$ (see Claim~\ref{claim:taking characters out of exceptional}, $\eta(z)^n=\eta(z)$ when $n$ is odd) and
$\eta\tau$ is $(\theta_{1,\gamma},\theta_{1,\gamma'})$-distinguished.
\end{proof}

This claim is also useful for the following observation.
Assume we have, for an admissible representation $\tau_1\otimes\tau_2\in\Alg{M_{n-k,k}}$ (i.e., $\tau_1\in\Alg{\GL_{n-k}}$,
$\tau_2\in\Alg{\GL_k}$ and both are admissible),
\begin{align*}
\Hom_{M_{n-k,k}}((\theta_{n-k}
\widetilde{\otimes}\theta_{k})\otimes(\theta'_{n-k}
\widetilde{\otimes}\theta'_{k}),\tau_1^{\vee}\otimes\tau_2^{\vee})\ne0.
\end{align*}
According to Corollary~\ref{corollary:hom of tensor separates}, if $k$ is even, $\tau_2$ is distinguished. In the
odd case, the corollary only implies that $\eta\tau_2$ is distinguished for some $\eta$ with $\eta^2=1$. Then by
Claim~\ref{claim:eta tau is dist when tau is for odd n}, $\tau_2$ is distinguished also when $n$ is odd. Similarly, we deduce that
$\tau_1$ is distinguished.

The following lemma is the application of the results of Section~\ref{section:Heisenberg jacquet modules} to the study of
\eqref{homspace:main for distinguished} and will be used repeatedly below.
\begin{lemma}\label{lemma:repeated filtration argument for 2 blockes}
Let $\theta=\theta_{n,1,\gamma}$, $\theta'=\theta_{n,1,\gamma'}$ and $\tau_1\otimes\tau_2\in\Alg{M_{n-k,k}}$ be an admissible representation. Assume
\begin{align}\label{eq:assumption dist 2 blocks}
\Hom_{M_{n-k,k}}((\theta\otimes\theta')_{U_{n-k,k}},\delta_{Q_{n-k,k}}^{1/2}\tau_1^{\vee}\otimes\tau_2^{\vee})\ne0.
\end{align}
Then for some $0\leq j\leq \min(n-k,k)$,
\begin{align*}
\Hom_{L_{n,k,j}}(\xi_{n,k,j}\otimes1,
(\delta_{Q_{n-k-j,j}}^{-1/2}(\tau_1)_{U_{n-k-j,j}}\otimes\delta_{Q_{j,k-j}}^{-1/2}(\tau_2)_{U_{j,k-j}})^{\vee})\ne0.
\end{align*}
Here $L_{n,k,j}=\GL_{n-k-j}\times\mathbb{GL}_j\times\GL_{k-j}$, where $\mathbb{GL}_j$ is the diagonal embedding of $\GL_j$ in $\mathrm{St}_{n-k,k}(\psi_j)$;
$\xi_{n,k,j}$ is the following representation of $p^{-1}(\GL_{n-k-j}\times \GL_{k-j})$,
\begin{align*}
&\xi_{n,k,j}=(\theta_{n-k-j,1,\gamma_1}\widetilde{\otimes}_{\gamma_{(j)}}\theta_{k-j,1,\gamma_2})\otimes
(\theta_{n-k-j,1,\gamma_1'}\widetilde{\otimes}_{\gamma'_{(j)}}\theta_{k-j,1,\gamma_2'}),
\end{align*}
where $\gamma_i$, $\gamma_i'$ are arbitrary, $\gamma_{(j)}$ is given by \eqref{eq:condition to determine gamma j k} with respect to $\gamma$ and the character $\psi$, and
$\gamma'_{(j)}$ is given by \eqref{eq:condition to determine gamma j k} with $\gamma'$ and $\psi^{-1}$; $1$ is the trivial character of $\mathbb{GL}_j$.
\end{lemma}
\begin{proof}[Proof of Lemma~\ref{lemma:repeated filtration argument for 2 blockes}]
By an analog of the Geometric Lemma of Bernstein and Zelevinsky (\cite{BZ2} Theorem~5.2 and \cite{BZ1} 5.9-5.12), in $\Alg{\widetilde{Q}_{n-k,k}}$,
$\theta$ is glued from
\begin{align*}
\ind_{\widetilde{\mathrm{St}}_{n,k}(\psi_j)U_{n-k,k}}^{\widetilde{Q}_{n-k,k}}(\theta_{U_{n-k,k},\psi_j}), \qquad 0\leq j\leq \min(n-k,k).
\end{align*}
Regarding the notation, see Section~\ref{section:Heisenberg jacquet modules}.
According to Lemma~\ref{lemma:nontwisted Jacquet of tensor of two is one}, $(\theta\otimes\theta')_{U_{n-k,k}}$ is glued in $\Alg{Q}_{n-k,k}$ from
\begin{align}\label{eq:filtration quotients}
\ind_{\mathrm{St}_{n,k}(\psi_j)U_{n-k,k}}^{Q_{n-k,k}}(\theta_{U_{n-k,k},\psi_j}\otimes
\theta'_{U_{n-k,k},\psi_j^{-1}}), \qquad 0\leq j\leq \min(n-k,k).
\end{align}
Note that $U_{n-k,k}$ acts trivially on \eqref{eq:filtration quotients} and we regard it as a representation in $\Alg{M_{n-k,k}}$. Also recall that
by Theorem~\ref{theorem:twisted Jacquet modules of theta along maximal parabolic},
\begin{align*}
&\theta_{U_{n-k,k},\psi_j}=\delta_{Q_{n-k-j,2j,k-j}}^{1/4}(\theta_{n-k-j,1,\gamma_1}\widetilde{\otimes}_{\gamma_{(j)}}\theta_{k-j,1,\gamma_2})\otimes(\theta_{2j,1,\gamma_3})_{U_{j,j},\psi_j},\\
&\theta'_{U_{n-k,k},\psi_j^{-1}}=\delta_{Q_{n-k-j,2j,k-j}}^{1/4}(\theta_{n-k-j,1,\gamma_1'}\widetilde{\otimes}_{\gamma'_{(j)}}\theta_{k-j,1,\gamma_2'})\otimes(\theta_{2j,1,\gamma_3'})_{U_{j,j},\psi_j^{-1}}.
\end{align*}
By Theorem~\ref{theorem:twisted Jacquet modules of theta along generic character in k=n-k case} and because $\gamma_{\psi}\gamma_{\psi^{-1}}=1$,
$(\theta_{2j,1,\gamma_3})_{U_{j,j},\psi_j}\otimes(\theta_{2j,1,\gamma_3'})_{U_{j,j},\psi_j^{-1}}$ is the trivial character of $\mathbb{GL}_j$.
Set $V=U_{n-k-j,j,j,k-j}$. Then
$\theta_{U_{n-k,k},\psi_j}\otimes\theta'_{U_{n-k,k},\psi_j^{-1}}\in\Alg{(L_{n,k,j}\ltimes V)}$ and is trivial on
$V$. For $g\in L_{n,k,j}$, let $\mathrm{mod}_V(g)$ be defined by
\begin{align*}
\int_{V}f(\rconj{g^{-1}}v)\ dv=\mathrm{mod}_V(g)\int_{V}f(v)\ dv.
\end{align*}
This is the modulus character from \cite{BZ2} (p.~444). Then $\mathrm{mod}_V(g)=\delta_{Q_{n-k-j,2j,k-j}}(g)$.
Therefore \eqref{eq:filtration quotients} becomes
\begin{align*}
\ind_{L_{n,k,j}V}^{Q_{n-k,k}}(\mathrm{mod}_V^{1/2}(\xi_{n,k,j}\otimes1)).
\end{align*}

Assumption~\eqref{eq:assumption dist 2 blocks} implies that for some $0\leq j\leq \min(n-k,k)$,
\begin{align*}
\Hom_{M_{n-k,k}}(\ind_{L_{n,k,j}V}^{Q_{n-k,k}}(\mathrm{mod}_V^{1/2}(\xi_{n,k,j}\otimes1)),\delta_{Q_{n-k,k}}^{1/2}\tau_1^{\vee}\otimes\tau_2^{\vee})\ne0.
\end{align*}
The \lhs\ equals
\begin{align}\nonumber
&\Hom_{M_{n-k,k}}(\tau_1\otimes\tau_2,\Ind_{L_{n,k,j}V}^{Q_{n-k,k}}(
\mathrm{mod}_V^{1/2}\delta_{Q_{n-k,k}}^{-1/2}(\xi_{n,k,j}\otimes1)^{\vee}))\\\nonumber&=
\Hom_{L_{n,k,j}}(\mathrm{mod}_V^{-1/2}\delta_{Q_{n-k,k}}^{1/2}(\tau_1\otimes\tau_2)_V,(\xi_{n,k,j}\otimes1)^{\vee})\\\nonumber&=
\Hom_{L_{n,k,j}}(\xi_{n,k,j}\otimes1,
(\delta_{Q_{n-k-j,j}}^{-1/2}(\tau_1)_{U_{n-k-j,j}}\otimes\delta_{Q_{j,k-j}}^{-1/2}(\tau_2)_{U_{j,k-j}})^{\vee}).\qedhere
\end{align}
\end{proof}

Our first objective is to describe how distinguished representations are constructed. In \cite{me9} we proved the following heredity result:
\begin{theorem}(\cite{me9} Theorem~3)\label{theorem:upper heredity of dist}
Let $\tau_1\otimes\tau_2\in\Alg{M_{n-k,k}}$. If $\tau_1$ and $\tau_2$ are distinguished,
$\tau_1\times\tau_2$ is distinguished. 
\end{theorem}
This does not exhaust all distinguished induced representations, as can already be observed in the case of $n=2$ and principal
series representations: if $\tau$ is any character of $F^*$, $\tau\times\tau^{-1}$ is distinguished. The proof is contained in \cite{me9} (Claim~4.4),
but was only reproduced from the arguments of Savin \cite{Savin3}. We extend this result to the following theorem.
\begin{theorem}\label{theorem:tau otimes tau dual is dist}
Let $\tau\in\Alg_{\mathrm{irr}}{\GL_k}$. Then $\tau\times\tau^{\vee}(=\Ind_{Q_{k,k}}^{\GL_{2k}}(\delta_{Q_{k,k}}^{1/2}\tau\otimes\tau^{\vee}))$ is a distinguished representation of $\GL_{2k}$.
\end{theorem}

The theorem essentially follows by exhibiting a functional in
\begin{align*}
\Hom_{M_{k,k}}((\theta\otimes\theta')_{U_{k,k}},\delta_{Q_{k,k}}^{1/2}\tau^{\vee}\otimes\tau).
\end{align*}
We will use the metaplectic Shalika model of $\theta$ obtained in
Theorem~\ref{theorem:twisted Jacquet modules of theta along generic character in k=n-k case}, 
to write an element of $(\theta\otimes\theta')_{U_{k,k}}$ as a locally constant function on $\GL_{2k}$, bounded by a Schwartz-Bruhat function.
The integral of this function against a matrix coefficient of $\tau$ will be absolutely convergent in some right half-plane, since it
can be bounded using the well-known zeta integral of Godement and Jacquet \cite{GJ}. Then one can obtain meromorphic continuation using
Bernstein's continuation principal (in \cite{Banks}).

\begin{proof}[Proof of Theorem~\ref{theorem:tau otimes tau dual is dist}]
Put $Q=Q_{k,k}$, $M=M_{k,k}$, $U=U_{k,k}$ and $G=\GL_k$.
Denote by $C^{\infty}_b(G)$ the space of complex-valued locally constant functions $f$ on $G$ that are bounded by a positive Schwartz-Bruhat function in $\mathcal{S}(F^{k\times k})$, that is,
there exists $\phi\in\mathcal{S}(F^{k\times k})$ such that $|f(g)|\leq \phi(g)$ for all $g$. The group $M$ acts on this space
by $\mathrm{diag}(a,b)f(x)=f(b^{-1}xa)$. Let $\theta=\theta_{2k,1,\gamma}$ and $\theta'=\theta_{2k,1,\gamma'}$ (the pair $\gamma$ and $\gamma'$ does not matter).
By Frobenius reciprocity, we need to prove
\begin{align}\label{claim:homspace need to prove}
\Hom_{M}((\theta\otimes\theta')_{U},\delta_{Q}^{1/2}\tau^{\vee}\otimes\tau)\ne0.
\end{align}

Denote $I(\psi)=\Ind_{\widetilde{\mathbb{G}}U}^{\widetilde{GL}_{2k}}(\gamma_{\psi,(-1)^{k-1}}\otimes\psi_k)$ and $i(\psi)=\ind_{\widetilde{\mathbb{G}}U}^{\widetilde{Q}}(\gamma_{\psi,(-1)^{k-1}}\otimes\psi_k)$, where $\mathbb{G}$ is the diagonal embedding of $G$ in $M$.
According to Theorem~\ref{theorem:twisted Jacquet modules of theta along generic character in k=n-k case} (see Remark~\ref{remark:about metaplectic Shalika model}), $\theta\subset I(\psi)$. Additionally $\theta(U)\subset\theta$ in $\Alg{\widetilde{Q}}$ and since $\ind_{\mathbb{\widetilde{\mathbb{G}}}U}^{\widetilde{Q}}(\theta_{U,\psi_k})\subset\theta(U)$, the theorem
also implies $i(\psi)\subset \theta$ (in $\Alg{\widetilde{Q}}$). When we apply
Theorem~\ref{theorem:twisted Jacquet modules of theta along generic character in k=n-k case} to $\theta'$ using $\psi^{-1}$, we obtain $i(\psi^{-1})\subset \theta'\subset I(\psi^{-1})$.
First we claim:
\begin{claim}\label{claim:embedding of tensor theta in space of bounded functions}
There exists $L\in \Hom_{M}((\theta\otimes\theta')_U,C^{\infty}_b(G))$, which restricts to
an isomorphism $(i(\psi)\otimes i(\psi^{-1}))_U\isomorphic\mathcal{S}(G)$ in $\Alg{M}$.
\end{claim}
Now let $\varphi\otimes\varphi'$ belong to the space of $\theta\otimes\theta'$. Also let $h$ be a matrix coefficient of $\tau$. For $s\in\C$, consider the integral
\begin{align}\label{int:formal integral}
\int_{G}L(\varphi\otimes\varphi')(g)h(g)\nu^{s-k/2}(g)\ dg.
\end{align}
Since $L(\varphi\otimes\varphi')\in C^{\infty}_b(G)$, there is some $s_0\in\R$ depending only on $\tau$, such that \eqref{int:formal integral}
is absolutely convergent for all $s$ with $\Re{(s)}>s_0$. This follows immediately from the convergence properties of the zeta integral
$\int_{G}\phi(g)h(g)\nu^s(g)dg$ of Godement and Jacquet (\cite{GJ} p.~30). In this right half-plane \eqref{int:formal integral} defines an element in
\begin{align}\label{homspace:for s}
\Hom_{M}((\theta\otimes\theta')_U,\delta_{Q}^{1/2}\nu^{-s}\tau^{\vee}\otimes\nu^{s}\tau).
\end{align}

Also since $L:(i(\psi)\otimes i(\psi^{-1}))_U\rightarrow\mathcal{S}(G)$ is an isomorphism, we can choose data $(\varphi,\varphi',h)$ such that
\eqref{int:formal integral} is absolutely convergent and equals $1$, for all $s$. In order to deduce meromorphic continuation, we also need to show that
\eqref{homspace:for s} is at most one-dimensional, except for a finite set of values of $q^{-s}$, where $q$ is the order of the residue field of $F$.
In fact, we prove the following more general statement.
\begin{claim}\label{claim:homspace at most one dimensional for a pair}
Let $\tau_1\otimes\tau_2\in\Alg_{\mathrm{irr}}{M_{n-k,k}}$. Outside of a finite set of values of $q^{-s}$,
\begin{align*}
\Hom_{M_{n-k,k}}((\theta\otimes\theta')_{U_{n-k,k}},\delta_{Q_{n-k,k}}^{1/2}\nu^{-s}\tau_1\otimes\nu^{s}\tau_2)
\end{align*}
is at most one-dimensional.
\end{claim}
Now according to Bernstein's principle of meromorphic continuation and rationality (in \cite{Banks}),
the integral~\eqref{int:formal integral} has a meromorphic continuation to a function $\Lambda(\varphi\otimes\varphi',h,s)$ in $\C(q^{-s})$, satisfying the same equivariance properties
with respect to $M$, and not identically zero at $s=0$. Assume $\Lambda(\cdot,\cdot,s)$ has a pole at $s=0$, of order $r\geq0$, then
$\lim_{s\rightarrow 0}s^r\Lambda(\cdot,\cdot,s)$ is finite and nonzero. Now \eqref{claim:homspace need to prove} follows and the theorem is proved.

\begin{proof}[Proof of Claim~\ref{claim:embedding of tensor theta in space of bounded functions}]
Regard $G(=\GL_k)$ as a subgroup of $\GL_{2k}$ via the embedding $\ell(x)=\mathrm{diag}(x,I_k)$.
In general, any $f\in I(\psi)$ defines a locally constant function on $\widetilde{G}$ by restriction. According to Lemma~3.1 of Friedberg and Jacquet \cite{FJ}, which can easily be verified in our setting, this function is bounded by a positive Schwartz-Bruhat function in $\mathcal{S}(F^{k\times k})$ (the more general statement in \cite{FJ} also holds, but will not be needed here).

Define a mapping $L:I(\psi)\otimes I(\psi^{-1})\rightarrow C^{\infty}_b(G)$
by
\begin{align*}
L(f\otimes f')(x)=f(\ell(x))f'(\ell(x)).
\end{align*}
Since
$\gamma_{\psi}\gamma_{\psi^{-1}}=1$, we have
$f(\mathrm{diag}(x,x)g)f'(\mathrm{diag}(x,x)g)=f(g)f'(g)$. Hence $L$ intertwines the action of $M$ and belongs to
\begin{align*}
\Hom_{M}(I(\psi)\otimes I(\psi^{-1}),C^{\infty}_b(G)).
\end{align*}
Observe that $L$ factors through $(I(\psi)\otimes I(\psi^{-1}))_U$. Indeed if $\sum_a f_a\otimes f'_a$ lies in the space of $(I(\psi)\otimes I(\psi^{-1}))(U)$, 
Lemma~\ref{lemma:Jacquet kernel as integral} implies in particular that for some compact $\mathcal{U}<U$,
\begin{align*}
(\sum_a f_a\otimes f'_a)(\ell(x))=\int_{\mathcal{U}}(\sum_a f_a\otimes f'_a)(\ell(x)u)\ du=0,\qquad\forall x\in G.
\end{align*}
Here we used the fact that for $x_0\in\widetilde{G}$ such that $p(x_0)=\mathrm{diag}(x,I_k)=\ell(x)$ ($p$ - the projection $\widetilde{\GL}_{2k}\rightarrow\GL_{2k}$, see Section~\ref{subsection:the groups}), $f_a(x_0\left(\begin{smallmatrix}I_k&u\\&I_k\end{smallmatrix}\right))=\psi_k(xu)f_a(x_0)$, and assumed $\int_{\mathcal{U}}du=1$.

Moreover, Lemma~\ref{lemma:nontwisted Jacquet of tensor of two is one} and its proof show that
$(i(\psi)\otimes i(\psi^{-1}))_U\isomorphic\ind_{\mathbb{G}U}^{Q}(1)\isomorphic\mathcal{S}(G)$ and the restriction
of $L$ to $i(\psi)\otimes i(\psi^{-1})\subset I(\psi)\otimes I(\psi^{-1})$ defines an isomorphism in
\begin{align*}
\Hom_{M}((i(\psi)\otimes i(\psi^{-1}))_U,\mathcal{S}(G)).
\end{align*}
The claim follows because 
$\theta\otimes\theta'\subset I(\psi)\otimes I(\psi^{-1})$.
\end{proof}

\begin{proof}[Proof of Claim~\ref{claim:homspace at most one dimensional for a pair}]
By virtue of Lemma~\ref{lemma:repeated filtration argument for 2 blockes} and using the same notation, it is enough to prove that for all
$0\leq j<\min(n-k,k)$, except for a finite set of $q^{-s}$,
\begin{align*}
\Hom_{L_{n,k,j}}(\xi_{n,k,j}\otimes1,
(\nu^{s}\delta_{Q_{n-k-j,j}}^{-1/2}(\tau_1^{\vee})_{U_{n-k-j,j}}\otimes\nu^{-s}\delta_{Q_{j,k-j}}^{-1/2}(\tau_2^{\vee})_{U_{j,k-j}})^{\vee})=0,
\end{align*}
and for $j=\min(n-k,k)$ this space is at most one-dimensional.

We show that this holds for each pair of irreducible subquotients $\varrho_1\otimes\varrho_2$ of $\delta_{Q_{n-k-j,j}}^{-1/2}(\tau_1^{\vee})_{U_{n-k-j,j}}$
and $\varrho_3\otimes\varrho_4$ of $\delta_{Q_{j,k-j}}^{-1/2}(\tau_2^{\vee})_{U_{j,k-j}}$. For these subquotients, this space vanishes unless both of the following spaces are nonzero:
\begin{align}\label{space:homspace xi 1}
&\Hom_{\GL_{n-k-j}\times\GL_{k-j}}(\xi_{n,k,j},\nu^{-s}\varrho_1^{\vee}\otimes\nu^{s}\varrho_4^{\vee}),\\
\label{space:homspace xi 2}
&\Hom_{\GL_j}(\nu^{s}\varrho_2,\nu^{s}\varrho_3^{\vee}).
\end{align}
If $j<k$ or $j<n-k$, Corollary~\ref{corollary:hom of tensor separates} (restricting \eqref{space:homspace xi 1}
to $\GL_{n-k-j}^{\square}\times\GL_{k-j}$ and $\GL_{n-k-j}\times\GL_{k-j}^{\square}$) implies that \eqref{space:homspace xi 1} vanishes unless
$\nu^{s}\varrho_1$ and $\nu^{-s}\varrho_4$ are distinguished, which is false for almost all values of $q^{-s}$ (simply considering the central characters).

The remaining case is when $j=k=n-k$, then \eqref{space:homspace xi 2} becomes
$\Hom_{\GL_k}(\nu^{s}\tau_1^{\vee},\nu^{s}\tau_2)$,
which is at most one-dimensional.
\end{proof}
\end{proof}

Theorems~\ref{theorem:upper heredity of dist} and \ref{theorem:tau otimes tau dual is dist}
imply the following corollary.
\begin{corollary}\label{corollary:correctness for irred generic}
Let $\alpha=(\alpha_1,\ldots,\alpha_m)$ be a partition of $n$ and $\tau_1\otimes\ldots\otimes\tau_m\in\Alg{M_{\alpha}}$. Assume that
if $\alpha_i$ is odd, $\tau_i$ is distinguished, and if $\alpha_i$ is even, $\tau_i$ is either distinguished, or
equals $\tau_{i,0}\times\tau_{i,0}^{\vee}$ for some $\tau_{i,0}\in\Alg_{\mathrm{irr}}{\GL_{\alpha_i/2}}$. Then
$\tau_1\times\ldots\times\tau_m$ is distinguished.
\end{corollary}

Here and onward we assume that $F$ is a field of characteristic $0$. This is because we rely on several results
not proven for fields of nonzero characteristic, e.g., certain
functoriality results.

The following remark will be used in our analysis of square-integrable distinguished representations.
\begin{remark}\label{remark:Shahidi's computation of L function L2 tau Sym2}
Let $\tau=\langle[\nu^{(1-l)/2}\rho,\nu^{(l-1)/2}\rho]\rangle^t$, where $\rho\in\Alg_{\mathrm{cusp}}{\GL_k}$ is unitary and $l\geq1$ is an
integer (see Section~\ref{subsection:representations}). Put $n=lk$.
The computation of $L(s,\tau,\mathrm{Sym}^2)$ by Shahidi \cite{Sh5} (Proposition~8.1) implies the following characterization for the existence of
a pole of $L(s,\tau,\mathrm{Sym}^2)$ at $s=0$:
\begin{enumerate}
\item If $n$ is odd and $\rho$ is self-dual, this is always the case.
\item When $n$ is even and $l$ is odd, $L(0,\tau,\mathrm{Sym}^2)=\infty$ if and only if $L(0,\rho,\mathrm{Sym}^2)=\infty$.
\item When $n$ and $l$ are even, $L(0,\tau,\mathrm{Sym}^2)=\infty$ if and only if $L(0,\rho,\bigwedge^2)=\infty$. 
\end{enumerate}
In particular, for a pole to exist $\rho$ (equivalently, $\tau$) must be self-dual.
\end{remark}

The next globalization lemma will be applied, to deduce that square-integrable representations $\tau$ such that
$L(s,\tau,\mathrm{Sym}^2)$ has a pole at $s=0$, are distinguished. It may also be
of independent interest. As mentioned in the introduction, there is an assumption here concerning the quasi-split case,
see Remark~\ref{remark:exact assumptions for globalization lemma} below.

\begin{lemma}\label{lemma:globalization lemma}
Let $\pi\in\Alg_{\mathrm{sqr}}{\GL_n}$. Assume that
$L(s,\pi,\mathrm{R})$ has a pole at $s=0$, for $\mathrm{R}=\mathrm{Sym}^2$ or $\bigwedge^2$. Then there exist
a number field with a ring of ad\`{e}les $\Adele$ and a global cuspidal representation
$\Pi$ of $\GL_n(\Adele)$, such that for a sufficiently large finite set of places $S$, $L^S(s,\Pi,\mathrm{R})$ has a pole at $s=1$, and at some place $v$, $\Pi_{v}=\pi$.
\end{lemma}
\begin{proof}[Proof of Lemma~\ref{lemma:globalization lemma}]
Consider the local functorial lift of \cite{CKPS}. Let
\begin{align*}
G_n=
\begin{cases}
\SO_{2m+1}&\text{$\mathrm{R}=\bigwedge^2$, $n=2m$,}\\
\Sp_{m}&\text{$\mathrm{R}=\mathrm{Sym}^2$, $n=2m+1$,}\\
\SO_{2m}&\text{$\mathrm{R}=\mathrm{Sym}^2$, $n=2m$.}
\end{cases}
\end{align*}
Here the even orthogonal group is quasi-split (split, or non-split but split over a quadratic extension). We claim that there is a generic square-integrable representation $\pi'$ of $G_n$, such
that
\begin{align}\label{eq:preservation of L and epsilon factors}
L(s,\pi'\times\varrho)=L(s,\pi\times\varrho),\qquad \epsilon(s,\pi'\times\varrho,\psi)=\epsilon(s,\pi\times\varrho,\psi),
\end{align}
for any generic $\varrho\in\Alg_{\mathrm{irr}}{\GL_k}$ and $k>0$. Here the $L$ and $\epsilon$-factors are defined by the Langlands-Shahidi method \cite{Sh3}.

Indeed, if $\mathrm{R}=\bigwedge^2$ this follows immediately from Theorem~2.1 of Jiang and Soudry \cite{JSd2}. Assume
$\mathrm{R}=\mathrm{Sym}^2$. Then if $\pi$ is supercuspidal, this follows from \cite{ACS} and \cite{PR} (the appendix). In fact,
the local functorial lift of an irreducible supercuspidal generic representation, of a quasi-split classical group, to some $\GL_N$ was detailed in
\cite{ACS} (Theorem~3.2). Given that, the results of Jiang and Soudry in the appendix of \cite{PR} show that the lift from quasi-split $SO_{2m}$
is surjective onto the supercuspidal representations of $\GL_{2m}$.

Regarding the remaining cases where $\mathrm{R}=\mathrm{Sym}^2$:
the existence of $\pi'$ when $n$ is odd, is contained in Theorem~4.8 of Liu \cite{BL}; if $n$ is even and the central character
$\omega_{\pi}=1$, this was proved in Theorem~4.8 of Jantzen and Liu \cite{JB} (in \cite{JB} only
the split case was analyzed, when $n=2m$ and $\omega_{\pi}\ne1$, $SO_{2m}$ will be non-split but quasi-split). We prove the last case in the following claim.
\begin{claim}\label{claim:quasi split lift}
A representation $\pi'$ with the above properties exists also when $n=2m$, $\mathrm{R}=\mathrm{Sym}^2$, $\pi$ is not supercuspidal and $\omega_{\pi}\ne1$.
\end{claim}

Additionally we take $\rho\in\Alg_{\mathrm{cusp}}{\GL_n}$ such that $L(s,\rho,\mathrm{R})$ has a pole at $s=0$. By
\cite{JSd2}, \cite{ACS} and \cite{PR} (the appendix), there is a generic supercuspidal representation $\rho'$ of $G_n$ satisfying equalities similar to
\eqref{eq:preservation of L and epsilon factors}.

Now according to Ichino, Lapid and Mao \cite{ILM} (Corollary~A.6, which actually holds for the types of $G_n$ stated above, when $G_n$ is split,
by replacing Proposition~A.5 with \cite{CKPS} Corollary~10.1), 
there exist a number field with a ring of
ad\`{e}les $\Adele$ and a cuspidal globally generic representation $\Pi'$ of $G_n(\Adele)$, such that
$\Pi'_{v_1}=\pi'$ and $\Pi'_{v_2}=\rho'$ for some pair of places $v_1$ and $v_2$.

Next we apply the global lift
of \cite{CKPS,CPSS} and obtain an automorphic representation $\Pi$ of $\GL_{n}(\Adele)$ such that $L^S(s,\Pi,\mathrm{R})$ has a pole at $s=1$,
assuming $S$ is large enough. Moreover, the global lift preserves the $\gamma$-factors at all finite places $v$ (\cite{CKPS} Proposition~7.2), hence
\begin{align*}
\gamma(s,\Pi_{v_i}\times\varrho,\psi)=
\gamma(s,\Pi'_{v_i}\times\varrho,\psi),\qquad\forall\varrho\in\Alg_{\mathrm{cusp}}{\GL_k},\quad 1\leq k\leq n-1,\quad i=1,2.
\end{align*}
Together with the fact that the local lift preserves $L$ and $\epsilon$-factors, we see that
\begin{align*}
&\gamma(s,\Pi_{v_1}\times\varrho,\psi)=
\gamma(s,\pi\times\varrho,\psi),\qquad
\gamma(s,\Pi_{v_2}\times\varrho,\psi)=
\gamma(s,\rho\times\varrho,\psi),
\end{align*}
for all $\varrho\in\Alg_{\mathrm{cusp}}{\GL_k}$ and $1\leq k\leq n-1$.
By virtue of the local converse theorem for $\GL_n$ (\cite{GH} after Theorem~1.1), this implies $\Pi_{v_1}=\pi$ and $\Pi_{v_2}=\rho$. Because $\rho$
is supercuspidal, we deduce that $\Pi$ is cuspidal.

\begin{proof}[Proof of Claim~\ref{claim:quasi split lift}]
We adapt the proofs from \cite{JSd2} (Theorem~2.1) and \cite{JB} (Theorem~4.8). We are interested in a simpler result, concerning only
square-integrable representations. Our assumptions on $\pi$ are that $n=2m$,
$L(s,\pi,\mathrm{Sym}^2)$ has a pole at $s=0$, $\pi$ is not supercuspidal and $\omega_{\pi}\ne1$. Hence
$\pi=\langle[\nu^{-l}\varepsilon,\nu^{l}\varepsilon]\rangle^t$, where $\varepsilon\in\Alg_{\mathrm{cusp}}{\GL_k}$ is self-dual, $0<l\in\mathbb{Z}$,
$k$ is even, $2m=(2l+1)k$ and $L(s,\varepsilon,\mathrm{Sym}^2)$ has a pole at $s=0$
(see Remark~\ref{remark:Shahidi's computation of L function L2 tau Sym2}).

According to \cite{ACS} and \cite{PR} (the appendix), there is a generic supercuspidal representation $\varepsilon'$ of a quasi-split $SO_{k}$ satisfying
\begin{align}\label{eq:preservation of L and epsilon factors 22}
L(s,\varepsilon'\times\varrho)=L(s,\varepsilon\times\varrho),\qquad \epsilon(s,\varepsilon'\times\varrho,\psi)=\epsilon(s,\varepsilon\times\varrho,\psi),
\end{align}
for any generic $\varrho\in\Alg_{\mathrm{irr}}{\GL_k}$ and $k>0$. Then
$\langle[\nu\varepsilon,\nu^{l}\varepsilon]\rangle^t\otimes\varepsilon'$ is a representation of $\GL_{lk}\times SO_k$. Consider the representation $\langle[\nu\varepsilon,\nu^{l}\varepsilon]\rangle^t\rtimes\varepsilon'$ of $SO_{2m}$, parabolically induced from $\langle[\nu\varepsilon,\nu^{l}\varepsilon]\rangle^t\otimes\varepsilon'$.
Let $\pi'$ be the unique irreducible generic subquotient
of $\langle[\nu\varepsilon,\nu^{l}\varepsilon]\rangle^t\rtimes\varepsilon'$. One can show that $\pi'$ is square-integrable
(here the root system of $SO_{2m}$ is of type $B_{m-1}$; the arguments of \cite{Mu4} Section~2 and \cite{Td}, for $SO_{2n+1}$, can be slightly modified
to apply to $SO_{2m}$;
in the notation of \cite{Mu4} the pair $(\varepsilon,\varepsilon')$ satisfies $(C1)$).
Now the proof of the preservation of local factors, i.e., \eqref{eq:preservation of L and epsilon factors}, follows exactly as in
\cite{JSd2} (proof of Theorem~2.1), using the multiplicativity of the $\gamma$-factors and \eqref{eq:preservation of L and epsilon factors 22}.
\end{proof}
\end{proof}
\begin{remark}\label{remark:exact assumptions for globalization lemma}
In one case, namely when $n=2m$, $\mathrm{R}=\mathrm{Sym}^2$ and $\omega_{\pi}\ne1$, the classical group $SO_{2m}$ is non-split (split over
a quadratic extension). In this case, the lemma is valid under two mild caveats.
First, the functorial lift in the quasi-split case was only described globally (\cite{CPSS}), leaving out the local lift, which we use.
Second, in appealing to \cite{ILM} (Corollary~A.6), we need the following result
obtained in \cite{CKPS} (Corollary~10.1) for split groups: 
for a globally generic cuspidal representation $\Pi'$ of $G_n(\Adele)$, at all places $v$, the local Langlands parameters of $\Pi'_v$ (the Satake parameters
when data are unramified) are bounded in absolute value by $\frac12-\frac1{N^2+1}$, where $\Pi'$ functorially lifts to a representation of
$\GL_N(\Adele)$. This result follows from the local lift and the strong bounds of Luo, Rudnick and Sarnak \cite{LRS}, and will be obtained once the details of the local lift are provided.
\end{remark}

We recall Theorem~2 of \cite{me10}:
\begin{theorem}\label{theorem:supercuspidal dist GLn and pole}
Let $\tau\in\Alg_{\mathrm{cusp}}{\GL_n}$ be unitary. Then $\tau$ is
distinguished if and only if $L(s,\tau,\mathrm{Sym}^2)$ has a pole at $s=0$.
\end{theorem}
We extend this result to square-integrable representations.
Note that in particular, a supercuspidal distinguished representation must be self-dual.
\begin{theorem}\label{theorem:L2 distinguished implies cusp support self-dual}
Let $\tau\in\Alg_{\mathrm{sqr}}{\GL_n}$. Then $\tau$ is distinguished if and only if
$L(s,\tau,\mathrm{Sym}^2)$ has a pole at $s=0$.
\end{theorem}
\begin{proof}[Proof of Theorem~\ref{theorem:L2 distinguished implies cusp support self-dual}]
Write $\tau=\langle[\nu^{(1-l)/2}\rho,\nu^{(l-1)/2}\rho]\rangle^t$, where $\rho\in\Alg_{\mathrm{cusp}}{\GL_k}$ is unitary 
and $l=n/k$. Assume that $\tau$ is distinguished. First we have the following claim regarding the combinatorial structure of $\tau$.

\begin{claim}\label{claim:combinatorial structure of distinguished L2}
If $\tau$ is distinguished, then $\rho$ is self-dual and when $l$ is odd, $\rho$ is distinguished.
\end{claim}
The claim implies that $\tau$ is self-dual, whence $L(s,\tau\times\tau)$ has a pole at $s=0$. If $n$ is odd, we immediately deduce $L(0,\tau,\mathrm{Sym}^2)=\infty$ (\cite{Sh5} Corollary~8.2, $L(0,\tau,\bigwedge^2)<\infty$ when $n$ is odd). If $n$ is even and
$l$ is odd, the claim and Theorem~\ref{theorem:supercuspidal dist GLn and pole} imply $L(0,\rho,\mathrm{Sym}^2)=\infty$,
hence according to the second case of Remark~\ref{remark:Shahidi's computation of L function L2 tau Sym2}, $L(0,\tau,\mathrm{Sym}^2)=\infty$.

Assume $n$ and $l$ are even and
suppose that $\tau$ is distinguished with $L(0,\tau,\mathrm{Sym}^2)<\infty$. Then $L(0,\rho,\bigwedge^2)<\infty$. The claim shows that $\rho$ is self-dual, whence $L(0,\rho,\mathrm{Sym}^2)=\infty$.

Let $G_n$
be the split odd general spin group of rank $\glnidx+1$ (for details see \cite{Asg,AsgSha,HS,me8}). In \cite{me8} we developed a theory of exceptional representations of $G_n$, parallel to the small representations of $\SO_{2n+1}$ of Bump, Friedberg and Ginzburg \cite{BFG,BFG2} (in the following argument one may take $G_n=\SO_{2n+1}$, the only loss is in a technical requirement, for the underlying local field to contain a $4$-th root of unity). These are representations of a double cover $\widetilde{G}_n$ of $G_n$ obtained by restricting
the cover of $\Spin_{2n+3}$ of Matsumoto \cite{Mats}. The definition of exceptional representations is similar to the definition for $\GL_n$,
and they also have a simple parametrization. If $\Theta$ and $\Theta'$ are
a pair of exceptional representations of $\widetilde{G}_n$, an admissible representation $\kappa$ of
$G_n$, which admits a central character, is called distinguished if $\Hom_{G_n}(\Theta\otimes\Theta',\kappa^{\vee})\ne0$.

In \cite{me10} we proved the following ``inflation" result.
Let $P_n$ be a maximal parabolic subgroup of $G_n$ with a Levi part isomorphic to $\GL_n\times\GL_1$.
For $\kappa_0\in\Alg{\GL_n}$ and $s\in\C$, define $\mathrm{I}(\kappa_0,s)=\Ind_{P_n}^{G_n}(\delta_{P_n}^{1/2}\kappa_0|\det{}|^{s}\otimes1)$.
According to \cite{me10} (Corollary~4.8), assuming that $\kappa_0\in\Alg_{\mathrm{irr}}{\GL_n}$ is tempered, $\mathrm{I}(\kappa_0,1/2)$ is distinguished for some pair $\Theta$ and $\Theta'$ (more specifically described in \cite{me10}, but this will not be needed here) if and only if $\kappa_0$ is distinguished.

We claim that $\mathrm{I}(\tau,1/2)$ is irreducible. Indeed according to
Asgari, Cogdell and Shahidi \cite{ACS} (Proposition~4.27), because $L(0,\rho,\mathrm{Sym}^2)=\infty$, $\mathrm{I}(\rho,1/2)$ is reducible
and for a real $s\ne\pm1/2$, $\mathrm{I}(\rho,s)$ is irreducible. Then Theorem~9.1(ii) of
Tadi\`{c} (\cite{Td2}, the proof is combinatorial in nature and carries over from $\SO_{2n+1}$ to $G_n$)
implies that $\mathrm{I}(\tau,1/2)$ is irreducible ($\mathrm{I}(\tau,1/2)$ is $\nu^{1/2}\delta(\rho,l)\rtimes1$ in the notation of \cite{Td2} and $l$ is even).

Being irreducible, $\mathrm{I}(\tau,1/2)$ is also generic.
Now the assumption that $\tau$ is distinguished implies that
$\mathrm{I}(\tau,1/2)$ is a quotient of a space $\Theta\otimes\Theta'$.
However, by Theorem~1 of \cite{me10}, as a representation of $G_n$ the space $\Theta\otimes\Theta'$ does not afford a Whittaker functional, contradiction.

In the opposite direction assume $L(0,\tau,\mathrm{Sym}^2)=\infty$. 
By virtue of Lemma~\ref{lemma:globalization lemma}, there exist a number field $\mathrm{k}$ with
a ring of ad\`{e}les $\Adele$ and a global cuspidal representation $\Pi$ of $\GL_n(\Adele)$, such that $L^S(1,\Pi,\mathrm{Sym}^2)=\infty$ and at some place $v$, $\Pi_v=\tau$. Bump and Ginzburg \cite{BG} proved that when
$L^S(1,\Pi,\mathrm{Sym}^2)=\infty$, the following period integral is nonvanishing (\cite{BG} Theorem~7.6)
\begin{align}\label{int:BG period}
\int_{Z\GL_n(\mathrm{k})\backslash\GL_n(\Adele)}\varphi_{\Pi}(g)\phi(g)\phi'(g)dg.
\end{align}
Here $Z$ is a subgroup of finite index in $C_{\GL_n(\Adele)}$ ($Z=C_{\GL_n(\Adele)}$ when $n$ is odd);
$\varphi_{\Pi}$ is a cusp form in the space of $\Pi$; $\phi$ and $\phi'$ are automorphic forms in the spaces of two global exceptional representations of $\widetilde{\GL}_n(\Adele)$ (defined in \cite{KP} Section~II). The local-global principle implies that $\tau$ is distinguished (see e.g. \cite{JR} Proposition~1).
\begin{proof}[Proof of Claim~\ref{claim:combinatorial structure of distinguished L2}]
If $l=1$, i.e., $\tau$ is unitary supercuspidal, the result follows from
Theorem~\ref{theorem:supercuspidal dist GLn and pole}. 
Assume $l>1$.
According to Zelevinsky \cite{Z3} (Proposition~9.6), $\delta_{Q_{n-k,k}}^{-1/2}(\tau^{\vee})_{U_{n-k,k}}=\tau_1^{\vee}\otimes\nu^{(1-l)/2}\rho^{\vee}$ with
$\tau_1=\langle[\nu^{(1-l)/2}\rho,\nu^{(l-3)/2}\rho]\rangle^t$.
Then
$\tau^{\vee}\subset\tau_1^{\vee}\times\nu^{(1-l)/2}\rho^{\vee}$. Since $\tau$ is distinguished,
\begin{align*}
\Hom_{\GL_n}((\theta_{n,1,\gamma}\otimes\theta_{n,1,\gamma'})_{U_{n-k,k}},\delta_{Q_{n-k,k}}^{1/2}\tau_1^{\vee}\otimes\nu^{(1-l)/2}\rho^{\vee})\ne0,
\end{align*}
for some pair $\gamma$ and $\gamma'$.
Hence by Lemma~\ref{lemma:repeated filtration argument for 2 blockes}, for some $0\leq j\leq \min(n-k,k)=k$,
\begin{align}\label{space:after lemma 2 blocs in eL2 proposition}
\Hom_{L_{n,k,j}}(\xi_{n,k,j}\otimes1,
(\delta_{Q_{n-k-j,j}}^{-1/2}(\tau_1)_{U_{n-k-j,j}}\otimes\delta_{Q_{j,k-j}}^{-1/2}(\nu^{(l-1)/2}\rho)_{U_{j,k-j}})^{\vee})\ne0.
\end{align}
Since $\rho$ is supercuspidal, we must have $j=0$ or $k$.

If $j=0$, the \lhs\ of \eqref{space:after lemma 2 blocs in eL2 proposition} becomes
\begin{align*}
\Hom_{\GL_{n-k}\times\GL_k}((\theta_{n-k,1,\gamma_1}\widetilde{\otimes}_{\gamma_{(0)}}\theta_{k,1,\gamma_2})\otimes
(\theta_{n-k,1,\gamma_1'}\widetilde{\otimes}_{\gamma'_{(0)}}\theta_{k,1,\gamma_2'}),
\tau_1^{\vee}\otimes\nu^{(1-l)/2}\rho^{\vee}).
\end{align*}
Corollary~\ref{corollary:hom of tensor separates} implies that this vanishes, unless $\nu^{(l-1)/2}\rho$ is distinguished. Since $\rho$ is unitary, this is only possible when $l=1$, but
we are assuming $l>1$.

Therefore we must have $j=k$. Then
$\delta_{Q_{n-k-j,j}}^{-1/2}(\tau_1)_{U_{n-2k,k}}=\tau_2\otimes\nu^{(1-l)/2}\rho$, $\tau_2=\langle[\nu^{(3-l)/2}\rho,\nu^{(l-3)/2}\rho]\rangle^t$, and the \lhs\ of \eqref{space:after lemma 2 blocs in eL2 proposition}
is nonzero only if both
\begin{align*}
&\Hom_{\GL_{n-2k}}(\theta_{n-2k,1,\gamma_1}\otimes
\theta_{n-2k,1,\gamma_1'},\tau_2^{\vee})\ne0,\qquad
\Hom_{\GL_{k}}(\rho,\rho^{\vee})\ne0.
\end{align*}
The second condition implies $\rho\isomorphic\rho^{\vee}$. 
The first implies that $\tau_2$ is distinguished, and because
$\tau_2\in\Alg_{\mathrm{sqr}}{\GL_{n-2k}}$, we can apply induction. If $l$ is odd, the induction terminates with a supercuspidal representation of $\GL_k$, namely $\rho$,
to which we apply Theorem~\ref{theorem:supercuspidal dist GLn and pole}.
\end{proof}
\end{proof}

\begin{remark}\label{remark:irred}
Let $\tau\in\Alg_{\mathrm{sqr}}{\GL_n}$. If $L(s,\tau,\mathrm{Sym}^2)$ has a pole at $s=0$, then $\mathrm{I}(\tau,s)$ (in the notation
of the proof of Theorem~\ref{theorem:L2 distinguished implies cusp support self-dual}) is reducible at $s=1/2$. Indeed,
the theorem above implies that $\tau$ is distinguished hence by \cite{me10} (Corollary~4.8),
$\mathrm{I}(\tau,1/2)$ is distinguished. If $\mathrm{I}(\tau,1/2)$ is irreducible, then it is also generic, contradicting the fact that
$\Theta\otimes\Theta'$ does not afford a Whittaker functional (\cite{me10} Theorem~1). Of course this is not a new observation, as
this reducibility already follows from the aforementioned result of Tadi\`{c} (\cite{Td2} Theorem~9.1).
\end{remark}

We are ready to prove the characterization result for generic distinguished representations.
\begin{theorem}\label{theorem:distinguished irred generic}
Let $\tau\in\Alg_{\mathrm{irr}}{\GL_n}$ be generic and write $\tau=\tau_1\times\ldots\times\tau_m$, where
$\tau_i\in\Alg_{\mathrm{esqr}}{\GL_{\alpha_i}}$ and $\alpha=(\alpha_1,\ldots,\alpha_m)$ is a partition
of $n$. Then $\tau$ is distinguished if and only if there is $0\leq m_0\leq\lfloor m/2\rfloor$ such that, perhaps after permuting
the indices of the inducing data, $\tau_{2i}=\tau_{2i-1}^{\vee}$ for $1\leq i\leq m_0$ and $\tau_i$ is distinguished for $2m_0+1\leq i\leq m$.
\end{theorem}
\begin{proof}[Proof of Theorem~\ref{theorem:distinguished irred generic}]
Since $\tau$ is irreducible, one may permute the inducing data $\tau_1\otimes\ldots\otimes\tau_m$ without affecting $\tau$. Hence one
direction follows immediately from Corollary~\ref{corollary:correctness for irred generic}.

Assume $\tau$ is distinguished.
The result is trivial for $m=1$, assume $m>1$. We may also assume (perhaps after applying a permutation) that $\alpha_i\leq \alpha_m$ for all $1\leq i\leq m$.
For each $1\leq i\leq m$, set
$\tau_i=\langle[\rho_i,\nu^{l_i-1}\rho_i]\rangle^t$, $l_i\geq1$, $\rho_i\in\Alg_{\mathrm{cusp}}{\GL_{\alpha_i/l_i}}$.
Also put $k=\alpha_m$ and $l=l_m$. By Zelevinsky \cite{Z3} (Proposition~9.6), for any $0\leq j\leq k$, $\delta_{Q_{j,k-j}}^{-1/2}(\tau_m)_{U_{j,k-j}}=0$ unless
$k/l$ divides $k-j$ and then
\begin{align*}
&\delta_{Q_{j,k-j}}^{-1/2}(\tau_m)_{U_{j,k-j}}=\langle[\nu^{a}\rho_m,\nu^{l-1}\rho_m]\rangle^t\otimes\langle[\rho_m,\nu^{a-1}\rho_m]\rangle^t,\qquad a=l(k-j)/k.
\end{align*}
Here $0\leq a\leq l$ and since $k/l$ divides $k-j$, $a\in\mathbb{Z}$. Note that if $a=l$ ($j=0$), $\langle[\nu^{a}\rho_m,\nu^{l-1}\rho_m]\rangle^t=1$, and when $a=0$ ($j=k$), $\langle[\rho_m,\nu^{a-1}\rho_m]\rangle^t=1$.

For $0\leq j\leq n-k$, write
\begin{align*}
&\delta_{Q_{n-k-j,j}}^{-1/2}(\tau_1\times\ldots\times \tau_{m-1})_{U_{n-k-j,j}}=\varrho_{1,j}\otimes\varrho_{2,j}\in\Alg{M_{n-k-j,j}}.
\end{align*}

Applying Lemma~\ref{lemma:repeated filtration argument for 2 blockes}, for some $0\leq j\leq \min(n-k,k)$,
\begin{align}\label{space:after lemma 2 blocs in general proposition}
\Hom_{L_{n,k,j}}(\xi_{n,k,j}\otimes1,
(\varrho_{1,j}\otimes\varrho_{2,j}\otimes\langle[\nu^{a}\rho_m,\nu^{l-1}\rho_m]\rangle^t\otimes\langle[\rho_m,\nu^{a-1}\rho_m]\rangle^t)^{\vee})\ne0.
\end{align}
In particular
\begin{align*}
\Hom_{\GL_{n-k-j}\times\GL_{k-j}}(&(\theta_{n-k-j,1,\gamma_1}\widetilde{\otimes}_{\gamma_{(j)}}\theta_{k-j,1,\gamma_2})\otimes
(\theta_{n-k-j,1,\gamma_1'}\widetilde{\otimes}_{\gamma'_{(j)}}\theta_{k-j,1,\gamma_2'}),\\&
\varrho_{1,j}^{\vee}\otimes(\langle[\rho_m,\nu^{a-1}\rho_m]\rangle^t)^{\vee})\ne0
\end{align*}
and according to Corollary~\ref{corollary:hom of tensor separates}, $\varrho_{1,j}$ and $\langle[\rho_m,\nu^{a-1}\rho_m]\rangle^t$ are distinguished.

If $j=0$, $\varrho_{1,j}=\tau_1\times\ldots\times\tau_{m-1}$ is a distinguished generic representation and $\langle[\rho_m,\nu^{a-1}\rho_m]\rangle^t=\tau_m$.
The proof is then complete, by applying induction to
$\tau_1\times\ldots\times\tau_{m-1}$.

Henceforth assume $j>0$ and then $a<l$. If also $j<k$, then since $\langle[\rho_m,\nu^{a-1}\rho_m]\rangle^t\in\Alg_{\mathrm{esqr}}{\GL_{k-j}}$ is distinguished, Theorem~\ref{theorem:L2 distinguished implies cusp support self-dual} implies
$\langle[\rho_m,\nu^{a-1}\rho_m]\rangle^t=\langle[\nu^{(1-a)/2}\rho,\nu^{(a-1)/2}\rho]\rangle^t$ for a self-dual $\rho\in\Alg_{\mathrm{cusp}}{\GL_{k/l}}$.
Put $A=(a-1)/2$. Then $\rho_m=\nu^{-A}\rho$,
\begin{align*}
&\tau_m=\langle[\nu^{-A}\rho,\nu^{-A+l-1}\rho]\rangle^t,\\
&\delta_{Q_{j,k-j}}^{-1/2}(\tau_m)_{U_{j,k-j}}=\langle[\nu^{A+1}\rho,\nu^{-A+l-1}\rho]\rangle^t\otimes\langle[\nu^{-A}\rho,\nu^{A}\rho]\rangle^t
\end{align*}
and note that $A+1\leq -A+l-1$ and $A\geq0$ (because $0<a<l$).
When $j=k$, we cannot deduce anything about the structure of $\tau_m$.

To proceed we need to describe the structure of $\varrho_{1,j}\otimes\varrho_{2,j}$. To this end we use the concrete description of the composition factors
of $\delta_{Q_{n-k-j,j}}^{-1/2}(\tau_1\times\ldots\times \tau_{m-1})_{U_{n-k-j,j}}$ given by Zelevinsky (\cite{Z3}~1.6), which is a
formulation of \cite{BZ2} (Theorem~5.2) for this setting. Let $\mathcal{B}$ be the set of $(m-1)\times2$ matrices whose coordinates are nonnegative integers,
the sum of the coordinates in the $i$-th row is $\alpha_i$, the sum of the coordinates in the first column is $n-k-j$, and the second column sums up to $j$.
Let $b\in\mathcal{B}$. The $(i,r)$-th coordinate of $b$ will be denoted $b(i,r)$.
For each $1\leq i\leq m-1$, 
$(\tau_i)_{U_{b(i,1),b(i,2)}}=0$ unless $\alpha_i/l_i$ divides $b(i,2)$ and then
\begin{align*}
&\delta_{Q_{b(i,1),b(i,2)}}^{-1/2}(\tau_i)_{U_{b(i,1),b(i,2)}}
=\langle[\nu^{a_i}\rho_i,\nu^{l_i-1}\rho_i]\rangle^t\otimes\langle[\rho_i,\nu^{a_i-1}\rho_i]\rangle^t,\qquad a_i=l_ib(i,2)/\alpha_i.
\end{align*}
Note that if $b(i,2)=\alpha_i$, $\langle[\nu^{a_i}\rho_i,\nu^{l_i-1}\rho_i]\rangle^t=1$ and when $b(i,2)=0$, $\langle[\rho_i,\nu^{a_i-1}\rho_i]\rangle^t=1$.
Assuming $\alpha_i/l_i$ divides $b(i,2)$ for each $1\leq i\leq m-1$, put
\begin{align*}
&\varrho_{1,j,b}=\langle[\nu^{a_1}\rho_1,\nu^{l_1-1}\rho_1]\rangle^t\times\ldots\times\langle[\nu^{a_{m-1}}\rho_{m-1},\nu^{l_{m-1}-1}\rho_{m-1}]\rangle^t,\\
&\varrho_{2,j,b}=\langle[\rho_1,\nu^{a_1-1}\rho_1]\rangle^t\times\ldots\times\langle[\rho_{m-1},\nu^{a_{m-1}-1}\rho_{m-1}]\rangle^t.
\end{align*}
Then $\delta_{Q_{n-k-j,j}}^{-1/2}(\tau_1\times\ldots\times \tau_{m-1})_{U_{n-k-j,j}}$ is glued from the representations
\begin{align*}
\varrho_{1,j,b}\otimes\varrho_{2,j,b},\qquad b\in\mathcal{B}.
\end{align*}
The following claim uses this description to analyze the structure of $\varrho_{1,j,b}$ and $\varrho_{2,j,b}$, given
that $\varrho_{2,j,b}$ contains a certain quotient.
\begin{claim}\label{claim:analysis of varrho 2 j}
Let $\pi\in\Alg_{\mathrm{cusp}}{\GL_{k/l}}$ and assume that $0<j\leq k$, $lj/k\in\mathbb{Z}$ and for some $b\in\mathcal{B}$,
\begin{align}\label{claim:assumption for analysis of varrho 2 j}
\Hom_{\GL_{j}}(\varrho_{2,j,b},\langle[\pi,\nu^{\frac{lj}k-1}\pi]\rangle^t)\ne0.
\end{align}
If $j<k$, then $\nu^{a_{i}-1}\rho_{i}=\nu^{\frac{lj}k-1}\pi$ for some $i$ with $b(i,2)>0$. If $j=k$, then for some $1\leq i\leq m-1$,
$\tau_i\isomorphic\langle[\pi,\nu^{l-1}\pi]\rangle^t$, and
\begin{align*}
\varrho_{1,k,b}=\tau_1\times\ldots\times\tau_{i-1}\times\tau_{i+1}\times\ldots\times\tau_{m-1}.
\end{align*}
\end{claim}

Now we can dispose of the nontrivial non-generic orbits. Indeed, assume $j<k$.
Suppose \eqref{space:after lemma 2 blocs in general proposition} does not vanish, then for some $b\in\mathcal{B}$,
\begin{align*}
\Hom_{\GL_{j}}(\varrho_{2,j,b},\langle[\nu^{A-l+1}\rho,\nu^{-A-1}\rho]\rangle^t)
\ne0.
\end{align*}
Applying Claim~\ref{claim:analysis of varrho 2 j} we deduce that $\nu^{a_i-1}\rho_{i}=\nu^{-A-1}\rho$ for some $i$ with $b(i,2)>0$.
Then
\begin{align*}
\tau_i=\langle[\nu^{-A-a_i}\rho,\nu^{-A-a_i+l_i-1}\rho]\rangle^t,\qquad\tau_m=\langle[\nu^{-A}\rho,\nu^{-A+l-1}\rho]\rangle^t.
\end{align*}
We show that $\tau_i$ and $\tau_m$ must be linked, contradicting the irreducibility of $\tau$. Specifically, we show
\begin{align*}
&-A-a_i<-A,\qquad -A-a_i+l_i-1<-A+l-1, \qquad-A-a_i+l_i-1\geq-A-1.
\end{align*}
Because
$\rho_i\in\Alg{\GL_{\alpha_i/l_i}}$ and $\rho\in\Alg{\GL_{k/l}}$, we have $\alpha_i/l_i=k/l$, and since
(by assumption) $\alpha_i\leq k$, we obtain $l_i\leq l$.
The first two inequalities follow from this and because $a_i=l_ib(i,2)/\alpha_i>0$. The last holds because
$b(i,2)\leq \alpha_i$ (by definition) whence $a_i\leq l_i$.

Finally, assume $j=k$. Then \eqref{space:after lemma 2 blocs in general proposition}
vanishes unless for some $b\in\mathcal{B}$, both
\begin{align*}
\Hom_{\GL_{n-2k}}(\theta_{n-2k,1,\gamma_1}\otimes
\theta_{n-2k,1,\gamma_1'},\varrho_{1,k,b}^{\vee})\ne0,\qquad 
\Hom_{\GL_{k}}(\varrho_{2,k,b},\tau_m^{\vee})\ne0.
\end{align*}
Again appealing to Claim~\ref{claim:analysis of varrho 2 j} we find that
$\tau_i\isomorphic\tau_m^{\vee}$ for some $1\leq i\leq m-1$ and
\begin{align*}
\varrho_{1,k,b}=\tau_1\times\ldots\times\tau_{i-1}\times\tau_{i+1}\times\ldots\times\tau_{m-1}.
\end{align*}
If $n=2k$, we are done. Otherwise the theorem follows by applying induction to $\varrho_{1,k,b}$.

\begin{proof}[Proof of Claim~\ref{claim:analysis of varrho 2 j}]
Let $1\leq i_1<\ldots<i_c\leq m-1$ be a minimal set of indices such that
$b(i_1,2)+\ldots+b(i_c,2)=j$, $c\geq1$. Then
\begin{align*}
\varrho_{2,j,b}=\langle[\rho_{i_1},\nu^{a_{i_1}-1}\rho_{i_1}]\rangle^t\times\ldots\times\langle[\rho_{i_c},\nu^{a_{i_c}-1}\rho_{i_c}]\rangle^t.
\end{align*}
In general if $\varepsilon\in\Alg{\GL_n}$ is a subrepresentation of a representation parabolically induced from a cuspidal representation
$\varepsilon_1\otimes\ldots\otimes\varepsilon_t\in\Alg_{\mathrm{cusp}}{M_{\beta}}$, where $\beta$ is a partition of $n$ into $t$ parts, we
say that the cuspidal support of $\varepsilon$ is the multiset (i.e., including multiplicities)
$\{\varepsilon_1,\ldots,\varepsilon_t\}$ and denote it by $\mathrm{Supp}(\varepsilon)$.
According to Zelevinsky \cite{Z3} (9.1 and 6.1)
\begin{align*}
\mathrm{Supp}(\varrho_{2,j,b})=\{\rho_{i_r},\ldots,\nu^{a_{i_r}-1}\rho_{i_r}\}_{1\leq r\leq c},\qquad
\mathrm{Supp}(\langle[\pi,\nu^{\frac{lj}k-1}\pi]\rangle^t)=\{\pi,\ldots,\nu^{\frac{lj}k-1}\pi\}.
\end{align*}
Assumption~\eqref{claim:assumption for analysis of varrho 2 j} implies that $\langle[\pi,\nu^{\frac{lj}k-1}\pi]\rangle^t$ is an irreducible subquotient of a representation
induced from the tensor of all the representations (with multiplicities) appearing in $\mathrm{Supp}(\varrho_{2,j,b})$. Hence by Bernstein and Zelevinsky \cite{BZ2} (Theorem~2.9),
as multisets
\begin{align}\label{eq:first eq of csup supports}
\mathrm{Supp}(\varrho_{2,j,b})=\{\pi,\ldots,\nu^{\frac{lj}k-1}\pi\}.
\end{align}

Therefore
$\rho_{i_r}=\nu^{s_{i_r}}\pi$ for some integer $s_{i_r}\geq 0$, for all $r$. Furthermore, there exists $i\in\{i_1,\ldots,i_c\}$
such that $\nu^{a_{i}-1}\rho_{i}=\nu^{\frac{lj}k-1}\pi$ and clearly $b(i,2)>0$. This completes the proof for $j<k$.

Henceforth assume $j=k$. If $c=1$, then for $i=i_1$ ($1\leq i\leq m-1$), $\alpha_i\geq b(i,2)=k$. Since $k\geq \alpha_i$, we see that $b(i,2)=\alpha_i$,
$\varrho_{2,k,b}=\tau_i$ and \eqref{claim:assumption for analysis of varrho 2 j} implies
$\tau_i\isomorphic\langle[\pi,\nu^{l-1}\pi]\rangle^t$. Also because $b(i,1)+b(i,2)=\alpha_i$ and the sum of entries in the second column of $b$ is $k$, we have $b(i,1)=0$, $b(r,2)=0$ and $b(r,1)=\alpha_r$ for all $r\ne i$. Thus
\begin{align*}
\varrho_{1,k,b}=\tau_1\times\ldots\times\tau_{i-1}\times\tau_{i+1}\times\ldots\times\tau_{m-1},
\end{align*}
as claimed.

Assume $c>1$. For simplicity, assume $i_r=r$ for $1\leq r\leq c$.
Now we can denote, for $1\leq i\leq c$,
\begin{align*}
[\rho_i,\nu^{l_i-1}\rho_i]=[s_i,e_i], \qquad [\rho_i,\nu^{a_i-1}\rho_i]=[s_i,h_i],\qquad s_i\leq h_i\leq e_i, \quad s_i,h_i,e_i\in\mathbb{Z}.
\end{align*}
Indeed, the representations $\rho_i$ are all unramified twists of one representation $\pi$, so $\rho_i$ can be safely excluded from the notation. In other words, the segment $\{\rho_i,\ldots,\nu^{l_i-1}\rho_i\}$ will henceforth be denoted $[s_i,e_i]$.
Note that $\tau_i=\langle[\rho_i,\nu^{l_i-1}\rho_i]\rangle^t=\langle[\nu^{s_i}\pi,\nu^{e_i}\pi]\rangle^t$. We can rewrite \eqref{eq:first eq of csup supports} in the form
\begin{align}\label{eq:second eq of csup supports}
\{\nu^{s_i}\pi,\ldots,\nu^{h_i}\pi\}_{1\leq i\leq c}=\{\pi,\ldots,\nu^{l-1}\pi\}.
\end{align}
In particular $0\leq s_i\leq l-1$ for all $i$.
Renumbering if necessary, we can assume that the segment $[s_1,e_1]$ is such that $s_1=0$ and $e_i\leq e_1$ for all $i$ such that $s_i=0$.

Because $\rho_i\in\Alg{\GL_{\alpha_i/l_i}}$ and $\pi\in\Alg{\GL_{k/l}}$, we have $\alpha_i/l_i=k/l$, and since
(by assumption) $\alpha_i\leq k$, we obtain $l_i\leq l$.
Since $e_i-s_i+1=l_i\leq l$ and $s_1=0$, $e_1\leq l-1$.
If $s_1<s_i\leq e_1+1$, then $s_i\leq e_i\leq e_1$ because otherwise the
segments $[s_1,e_1]$ and $[s_i,e_i]$ are linked contradicting the fact that $\tau$ is irreducible. In particular, $s_i\ne e_1+1$ for all $i$. If
$e_1<l-1$, $\nu^{e_1+1}\pi$ must appear in $\mathrm{Supp}(\varrho_{2,j,b})$, but this is impossible because $e_1+1$ is not covered by any segment
$[s_i,e_i]$. Therefore $e_1=l-1$, then $s_i\leq e_1$, whence $e_i\leq e_1$ for all $i$.

Thus far we have shown $[s_1,e_1]=[0,l-1]$, which immediately implies
$\tau_1=\langle[\pi,\nu^{l-1}\pi]\rangle^t$, and all other segments are contained in the first segment. Next we argue that the segments $[s_i,h_i]$
are disjoint. Indeed, any such segment contributes at most $h_i-s_i+1=a_i=l_ib(i,2)/\alpha_i$ representations to $\mathrm{Supp}(\varrho_{2,j,b})$
and since $l_i/\alpha_i=l/k$,
\begin{align*}
\sum_{i=1}^ch_i-s_i+1=\frac lk\sum_{i=1}^cb(i,2)=l.
\end{align*}
But $\mathrm{Supp}(\langle[\pi,\nu^{l-1}\pi]\rangle^t)$ contains $l$ non-isomorphic representations, hence by \eqref{eq:second eq of csup supports} the
segments are disjoint.

Therefore, perhaps after renumbering segments $2,\ldots c$,
we can assume $s_1<s_2\ldots<s_c$. Then the fact that their lengths sum up to $l$ also implies $s_{i+1}=h_i+1$ whence
\begin{align*}
s_1\leq h_1<s_2\leq h_2<s_3\ldots<s_{c}\leq h_c.
\end{align*}
Because the segments $[s_i,e_i]$, $[s_{i+1},e_{i+1}]$ are not linked and $s_i<s_{i+1}\leq e_i+1$ ($h_i\leq e_i$), we obtain
$e_c\leq\ldots\leq e_2\leq e_1\leq l-1$. Now we must have $e_c=l-1$, because $\nu^{l-1}\pi\in \mathrm{Supp}(\varrho_{2,j,b})$ only if
$h_c\geq l-1$. It follows that $h_c=e_c$ and $e_1=\ldots=e_c$.

We conclude $[h_i+1,e_i]=[s_{i+1},e_{i+1}]$ for all $1\leq i\leq c-1$ and 
\begin{align*}
\langle[\nu^{h_1+1}\pi,\nu^{e_1}\pi]\rangle^t\times\ldots\times\langle[\nu^{h_{c-1}+1}\pi,\nu^{e_{c-1}}\pi]\rangle^t=\langle[\nu^{s_2}\pi,\nu^{e_2}\pi]
\rangle^t\times\ldots\times\langle[\nu^{s_c}\pi,\nu^{e_c}\pi]\rangle^t=\tau_2\times\ldots\times\tau_c.
\end{align*}
For the last segment $[s_c,h_c]=[s_c,e_c]$, so $b(c,1)=0$ and
\begin{align*}
\varrho_{1,j,b}=\langle[\nu^{h_1+1}\pi,\nu^{e_1}\pi]\rangle^t\times\ldots\times\langle[\nu^{h_{c-1}+1}\pi,\nu^{e_{c-1}}\pi]\rangle^t\times\tau_{c+1}\times\ldots\times\tau_{m-1}
=\tau_2\times\ldots\times\tau_{m-1}.
\end{align*}
Together with $\tau_1=\langle[\pi,\nu^{l-1}\pi]\rangle^t$ (proved above), the proof is complete also when $j=k$ and $c>1$. Note that
the indices were permuted during the argument, in the original representation $\tau_1$ will actually be $\tau_i$ for some $1\leq i\leq m-1$.
\end{proof}
\end{proof}
\begin{remark}\label{remark:extending combinatorial theorem for any field}
The only place in the proof where we assumed that the characteristic of the field is $0$, was in the application
of Theorem~\ref{theorem:L2 distinguished implies cusp support self-dual}. However, it is simple to see that we only
used the fact that a square-integrable distinguished representation is self-dual. Granted this, Theorem~\ref{theorem:distinguished irred generic} is valid also for fields of odd characteristic.
\end{remark}
\begin{corollary}
Let $\tau\in\Alg_{\mathrm{irr}}{\GL_n}$ be a generic distinguished representation. Then $\tau$ is self-dual.
\end{corollary}


\bibliographystyle{alpha}

\end{document}